\documentclass[12pt]{amsart}

\usepackage[english]{babel}
\usepackage{graphicx}

\usepackage[colorlinks= True, citecolor=blue]{hyperref}
\usepackage{amsthm}
\usepackage{amsfonts}
\usepackage{amssymb}
\usepackage{amsmath}
\usepackage{mathrsfs}
\usepackage{mathtools}
\usepackage[margin=1in]{geometry}

\usepackage[sort&compress,numbers]{natbib}

\usepackage{xcolor}
\usepackage{xcolor}

\usepackage{enumerate}

\newcommand{\bip}{{\operatorname{BP}}}
\newcommand{\creg}{{\operatorname{C}}}
\newcommand{\ireg}{{\operatorname{I}}}

\makeatletter
\def\namedlabel#1#2{\begingroup
	\def\@currentlabel{#2}%
	\phantomsection\label{#1}\endgroup
}
\makeatother

\newcommand{\ORD}{\operatorname{ORD}}
\allowdisplaybreaks

\newcommand{\bzeta}{\boldsymbol{\zeta}}

\newcommand{\bzer}{\boldsymbol{0}}

\newcommand{\cM}{{\mathcal{M}}}

\newcommand{\bcM}{{\boldsymbol{\cM}}}

\newcommand{\bz}{{\mathbf{z}}}

\makeatletter
\def\namedlabel#1#2{\begingroup
	#2%
	\def\@currentlabel{#2}%
	\phantomsection\label{#1}\endgroup
}
\makeatother

\newcommand{\bS}{{\mathbf{S}}}

\newtheorem{theorem}{Theorem}[section]
\newtheorem{corollary}[theorem]{Corollary}
\newtheorem{lemma}[theorem]{Lemma}
\newtheorem{proposition}[theorem]{Proposition}

\theoremstyle{definition}
\newtheorem{definition}[theorem]{Definition}

\newtheorem{regime}{Regime}
\newtheorem{assumption}[theorem]{Assumption}

\newtheorem{remark}[theorem]{Remark}

\DeclareMathAlphabet{\mathscrbf}{OMS}{mdugm}{b}{n}

\numberwithin{equation}{section}

\newcommand{\dotms}{\dotsm}

\newcommand{\Leb}{\operatorname{Leb}}

\newcommand{\by}{{\mathbf y}}
\newcommand{\bx}{{\mathbf x}}

\newcommand{\bw}{{\mathbf w}}
\newcommand{\bW}{{\mathbf W}}

\newcommand{\bv}{{\mathbf v}}
\newcommand{\bu}{{\mathbf u}}

\newcommand{\bDelta}{\boldsymbol{\Delta}}

\newcommand{\eps}{\varepsilon}
\newcommand{\R}{\mathbb{R}}

\newcommand{\D}{{\mathbb{D}}}

\newcommand{\N}{\mathbb{N}}

\newcommand{\cY}{\mathcal{Y}}
\newcommand{\bX}{{\mathbf X}}
\newcommand{\bT}{{\mathbf T}}

\newcommand{\weakarrow}{{\overset{(d)}{\Longrightarrow}}}

\newcommand{\PR}{\mathbb{P}}

\newcommand{\bbX}{{\mathbb{X}}}

\newcommand{\G}{{\mathcal{G}}}
\newcommand{\cI}{\mathcal{I}}

\newcommand{\btheta}{\boldsymbol{\theta}}

\newcommand{\SBM}{{\operatorname{SBM}}}
\newcommand{\cC}{{\mathcal{C}}}

\newcommand{\len}{\operatorname{len}}

\newcommand{\EE}{\mathcal{E}}
\newcommand{\LL}{\mathcal{L}}

\newcommand{\bccc}{\mathbf{c}}

\newcommand{\cS}{\mathcal{S}}

\newcommand{\fS}{\mathfrak{S}}

\newcommand{\Exp}{\textup{Exp}}
\title{Component sizes of rank-2 multiplicative random graphs}
\author{David Clancy, Jr.}
\date{\today}
\usepackage{verbatim}
\usepackage{todonotes}

\begin{document}
	
	\maketitle
	
	\begin{abstract}
		We show that in three different critical regimes, the masses of the connected components of rank-2 multiplicative random graph converge to lengths of excursions of a thinned L\'{e}vy process, perhaps with random coefficients. The three critical regimes are those identified by Bollob\'{a}s, Janson and Riordan \cite{BJR.07}, the interacting regime identified by Konarovskyi and Limic \cite{KL.21}, and what we call the nearly bipartite regime which has recently gained interest for its connection to random intersection graphs. Our results are able to extend some of the results by Baslingker et al.\cite{BBBSW.23} on component sizes of the stochastic blockmodel with two types and those of Federico \cite{Federico.19} and Wang \cite{Wang.23} on the sizes of the connected components of random intersection graphs. 
	\end{abstract}

	\section{Introduction}
	
	Over the last several decades, many models of random graphs have been proposed and studied. They often undergo a phase transition described informally as follows. One can parameterize the edge density by a parameter $t>0$ and find a critical value $t_c$ such that as the number of vertices $n\to\infty$ all the connected component are microscopic (typically $O(\log (n))$ in size) for $t<t_c$ while for $t>t_c$ there is a single macroscopic component of size $\Theta(n)$. This property holds for a wide class of random graph models \cite{BJR.07} and has prompted a lot of interest in understanding the sizes (and subsequently the geometry) of the connected components at and near the critical value $t_c$.  
	
	A major breakthrough in understanding the component sizes of random graphs at criticality came with Aldous's work on the Erd\H{o}s-R\'{e}nyi random graph \cite{Aldous.97}. Using elementary weak convergence techniques he was able to show that if you take the Erd\H{o}s-R\'{e}nyi random graph $G(n,n^{-1}+\lambda n^{-4/3})$ and list the connected components $(\cC_l^{(n)};l\ge 1)$ in decreasing order of their size then there exists some limiting random vector $\bzeta^{1,0,\lambda}$ such that
	\begin{equation*}
		\left(n^{-2/3}\#\cC_l^{(n)};l\ge 1\right) \weakarrow \bzeta^{1,\bzer,\lambda}\qquad\textup{ in }\ell^2.
	\end{equation*}
	A precise description of the limiting object is delayed until Section \ref{sec:prelims}.

	This weak convergence approach has proved to be quite successful. In \cite{AL.98} Aldous and Limic identified the entrance boundary of the multiplicative coalescent (see \cite{Aldous.97}) which is a key object that describes the connected sizes of connected components of many random graphs. See \cite{DLV.19,BSW.17, BBSW.14,BS.20,DvdHvLS.17,DvdHvLS.20,BDvdHS.20,CKG.20,Joseph.14,Wang.23,Federico.19,KL.21,CKL.22,AL.98} for a non-exhaustive list of such works. However, it has mostly been used to understand the scaling limits of so-called rank-1 models of random graphs where the expected adjacency matrix is (approximately) rank-1 \cite{CRvdH.22}. Unfortunately, many random graph models of interest to computer scientists, statisticians and physicists are not rank-1. One such model is the stochastic blockmodel (SBM) \cite{HLL.83,Abbe:2017} and its closely related inhomogeneous counterpart the degree-corrected SBM (DCSBM) \cite{Karrer:2011}.

	To put things on a more solid foundation, consider the following model as in \cite{CRvdH.22}. Let $A = A_n\in \R^{n\times n}_+$ be a matrix with non-negative entries and let $t>0$ be a parameter. The graph $G_n = G_n(A,t)$ on the vertex set $[n]:=\{1,\dotsm,n\}$ with edges included independently with probabilities
	\begin{equation*}
		\PR(i\sim j) = 1-\exp(-t a_{i,j})
	\end{equation*}
	is a rank-$k$ random graph whenever $\operatorname{rank}(A) = k$. While the rank-1 case is very well understood, much is still open about even the simplest extension to rank-2 random graphs. Quoting from \cite{CRvdH.22}
	\begin{quote}
		The general structure in Bollob\'{a}s, Janson and Riordan \cite{BJR.07} allows one to identify exactly when there is a giant or not, as well of several other properties of the random graphs involved. However, this set-up is likely not suitable to decide what the precise critical behavior is. Already in the rank-1 case, there are several universality classes (see \cite[Chapter 4]{vanderHofstad.23} and the references therein for an extensive overview), and one can expect that
		this only becomes more tricky when the rank is higher.
	\end{quote}
	
	We show that this expectation that things become more tricky is essentially \textit{not} the case for the component sizes in the rank-2 setting under mild assumptions generalizing those found in \cite{AL.98}. This should be compared with results by Miermont \cite{Miermont.08} and Berzunza \cite{BerzunzaOjeda.18} on multitype Bienaym\'{e}-Galton-Watson trees convergening stable continuum random trees. Key to our approach is the recent work by Konarovskyi, Limic and the author initiated in \cite{CKL.22} on the degree-corrected stochastic blockmodel (DCSBM).

	\subsubsection*{Organization of the paper} 
	In Section \ref{sec:mainResutlsStated}, we state our main results after we recall the necessary rank-1 results needed to properly state our main results. Some of these results (Theorem \ref{thm:classic} and \ref{thm:bipartite}) imply other results in the literature, and we take the time to draw these connections explicitly in Section \ref{sec:discussion}. In particular, we are able to extend recent results of Baslingker et al. in \cite{BBBSW.23} for SBMs to degree corrected SBMs of \cite{Karrer:2011}. We are also able to extend the results of Federico and Wang \cite{Federico.19,Wang.23} to the component sizes of inhomogeneous random intersection graphs.

	In Section \ref{sec:exploration}, we use the construction in \cite{CKL.22} (and a slight extension) to construct several relevant processes that encode the component sizes of the random graphs. We see that for each of the three regimes we have a pair of stochastic process $(V, U^2\circ X^{2,1})$ which encode the component sizes of the random graph. 
	
	In Section \ref{sec:generaltopology}, we recall some necessary results on the $M_1$ topology that allow us to efficiently establish the weak convergence of the encoding processes using recent results of Limic \cite{Limic.19,Limic.19_Supplement}. These weak convergence arguments are found Section \ref{sec:WeakConv}. Finally, in Section \ref{sec:proofsofMain} we prove our main results.
	
	As noted above, we prove our weak convergence results using the Skorohod $M_1$ topology instead of the standard $J_1$ topology. This is due to needing to use several functionals that are discontinuous in the $J_1$ topology but continuous in the $M_1$ topology -- at least on the relevant subsets that the limits a.s. live in. Due to the use of the $M_1$ topology, we have a lengthy Section \ref{sec:M1} on the useful $M_1$-continuous functionals.
	
	\subsection*{Acknowledgements} The author was partially supported by NFS-DMS 2023239. He would also like to thank Vitalii Konarovskyi and Vlada Limic for helpful discussions early on in hte project.

	\section{Main Results} \label{sec:mainResutlsStated}
	
	\subsection{Component sizes of rank-1 models}\label{sec:prelims}

	For $p\in[1,\infty)$ write $\ell^p_\downarrow$ denote the collection of vectors $\bz = (z_1,z_2,\dotsm)\in\R^\infty_+$ where $z_1\ge z_2\ge\dotsm\ge 0$ and
	\begin{equation*}
		\sigma_p(\bz):=\sum_{j=1}^\infty z_j^p = \|\bz\|_{\ell^p}^p <\infty.
	\end{equation*} Write $\ell^\infty_\downarrow$ for the non-negative decreasing sequences of real numbers. We will write $\ell^f_\downarrow$ for the collection of finite length vectors $\bz = (z_1,\dotms, z_k,0,\dotsm)$ where $z_k>0$ and we will write $\len(\bz) = k$ as the length of the vector $\bz$. Given any vector $\bz = (z_1,z_2,\dotsm)\in\ell^p$, $1\le p<\infty$ of non-negative real numbers we denote $\ORD(\bz)$ the unique element of $\ell^p_\downarrow$ obtained by rearranging the entries of $\bz$. We will write $[n]:=\{1,2,\dotsm,n\}$.
	
	Given a vector $\bz\in \ell^f_\downarrow$, and a scalar $q\ge 0$ we will write $\G(\bz,q)$ as the rank-1 multiplicative random graph on $[\len(\bz)]$ where each edge $i,j$ is included independently with probability
	\begin{equation*}
		\PR(i\sim j\textup{ in }\G(\bz,q)) = 1-\exp(-qz_iz_j).
	\end{equation*}
	It is easy to see that one can couple these graph $(\G(\bz,q);q\ge 0)$ where $\G(\bz,q)\subset \G(\bz,q')$ for all $q\le q'$. Given a set $A\subset \G(\bz,q)$ we will write $\cM(A) = \sum_{i\in A} z_i$ and will call this the total mass of the set $A$.

	The random graph was studied by Aldous and Limic \cite{AL.98}, extending the prior work of \cite{Aldous.97} on the Erd\H{o}s-R\'{e}nyi random graph. As identified in \cite{AL.98}, the limiting component sizes of the near-critical random graph $\G(\bz,q)$ are determined by the following asymptotic conditions on the sequence of weight vectors $\bz  = \bz^{(n)}$ and $q = q^{(n)}$:
	\begin{subequations}
		\begin{align}
			\label{eqn:sigma2tozer}&\sigma_2(\bz) \longrightarrow 0 \\
			\label{eqn:sigma3order} &\frac{\sigma_3(\bz)}{\sigma_2(\bz)^3} \longrightarrow \beta + \sum_{j=1}^\infty \theta_j^3 <\infty\\
			\label{eqn:sigmahumbs}
			&\forall j\ge 1\quad \frac{z_j}{\sigma_2(\bz)} \longrightarrow \theta_j\\
			\label{eqn:qrange} &q = \frac{1}{\sigma_2(\bz)} + \lambda
		\end{align}
	\end{subequations}where $\btheta = (\theta_1,\theta_2,\dotsm)\in \ell^3_\downarrow$, $\beta\ge 0$ and $\lambda\in \R$. As this will play a fundamental role in the sequel, we will write
	\begin{equation*}
		\cI^\circ  = \R_+\times \ell^3_\downarrow\qquad \textup{and}\qquad \cI = \cI^\circ \setminus \{0\}\times \ell^2_\downarrow,
	\end{equation*}
	and we will write 
	\begin{equation*}
		\bz\rightsquigarrow (\beta,\btheta)\textup{  if and only if \eqref{eqn:sigma2tozer}--\eqref{eqn:sigmahumbs} hold}.
	\end{equation*} Whenever we write $\bz\rightsquigarrow(\beta,\btheta)$ will implicitly assume that $(\beta,\btheta)\in \cI^\circ$.
	
	The limiting component sizes are encoded by a \textit{thinned L\'{e}vy process} described as follows. Given any $\bccc\in \ell^3_\downarrow$, let
	\begin{equation}\label{eqn:Jcreg}
		J^\bccc(t) = \sum_{p\ge 1} c_p (1_{[\xi_p\le t]} - c_p t)
	\end{equation}
	where $\xi_p\sim \Exp(c_p)$ are independent exponential random variables with rates $c_p$ (and means $1/c_p$). Let $W$ be an independent standard Brownian motion then
	\begin{align}\label{eqn:thinnedLevy}
		W^{\beta,\btheta, \lambda}(t)&= \sqrt{\beta} W(t) + \lambda t - \frac{\beta}{2}t^2 + J^{\btheta}(t).
	\end{align} Whenever $\beta>0$ or $\btheta\in \ell^3_\downarrow\setminus \ell^2_\downarrow$, Aldous and Limic \cite{AL.98} show that one can write 
	\begin{equation*}
		\left\{t: W^{\beta,\btheta,\lambda} (t)> \inf_{s\le t} W^{\beta,\btheta,\lambda}(s)\right\} = \bigsqcup_{j=1}^\infty (l_j,r_j) 
	\end{equation*} where $r_1-l_1\ge r_2-l_2\ge\dotsm> 0$. Set  
	\begin{equation}\label{eqn:MCextremedef}
		\bzeta^{\beta,\btheta,\lambda} = \left(r_j-l_j;j\ge 1\right).
	\end{equation}
	\begin{theorem}[{Aldous and Limic \cite{AL.98}}]
		Let $(\cC_j^{(n)};j\ge 1)$ be the connected components of $\G(\bz^{(n)},q)$ listed in decreasing order of mass. If $\bz^{(n)}\rightsquigarrow (\beta,\btheta)\in \cI$ and \eqref{eqn:qrange} holds then 
		\begin{equation*}
			\left(\cM(\cC_j^{(n)});j\ge 1\right)\weakarrow \bzeta^{\beta,\btheta,\lambda}\qquad\textup{ in }\quad \ell^2_\downarrow.
		\end{equation*}
	\end{theorem}
	
	More generally, whenever $Z$ is a c\`adl\`ag stochastic process such that we can a.s. write
	\begin{equation*}
		\left\{t : Z(t)>\inf_{s\le t} Z(s)\right\} = \bigsqcup_{j=1}^\infty (l_j^Z,r_j^Z)
	\end{equation*}
	where $r_1^Z-l_1^Z\ge r^Z_2-l_2^Z\ge\dotsm\ge 0$ (and ties broken in some measurable way), we will let 
	\begin{equation}\label{eqn:EE_inftyDef}   \LL^\downarrow_\infty (Z) = \left(r_j^Z-l_j^Z;j\ge 1 \right)\in \ell^\infty_\downarrow,\quad\qquad\EE_\infty(Z) = \{\{(l_j,r_j);j\ge 1\}\}
	\end{equation}
	where we write $\{\{\cdot\}\}$ for a multiset. If we wish to clear up confusion with additional paramters appearing, we will write $\LL^\downarrow_\infty(Z(s))$ in lieu of $\LL^\downarrow_\infty(Z)$.

	\subsection{Rank-2 model}
	In this work we will be interested in the rank-2 random graph denoted by $\G(\bW,Q)$. Here $\bW = (\bw^1,\bw^2)$ with $\bw^i\in \ell^f_\downarrow$ are two weight vectors and $Q\in \R^{2\times 2}_+$ is a symmetric $2\times 2$ matrix with non-negative entries. Each vector $\bw^i = (w_l^i;l\ge 1)$. The graph $\G(\bW,Q)$ is on $\len(\bw^1)+\len(\bw^2)$ many vertices made of two types $i\in\{1,2\}$. The type $i$ vertices are indexed by $(l,i)$ for $l\in[\len(\bw^i)]$. Edges are included independently with probabilities
	\begin{equation*}
		\PR\left((l,i)\sim (r,j) \textup{ in }\G(\bW,Q)\right) = 1-\exp\left(- q_{i,j} w^i_l w^j_r\right).
	\end{equation*}
	This graph was studied by Konarovskyi, Limic and the author in \cite{CKL.22} and is an example of the degree-corrected SBM introduced in \cite{Karrer:2011}. One can easily see that $\G(\bW,Q) = G_n(A,t)$ from the introduction and that $A$ is a rank-2 matrix. 
	
	A standing assumption throughout the article is the following:
	\begin{assumption}\label{ass:main}
		The weight parameters $\bw^{i} = \bw^{i,(n)}$ satisfy $\bw^i\rightsquigarrow (\beta^i,\btheta^i)$ for each $i\in[2]$.
		Moreover, 
		\begin{equation}\label{eqn:Qextension}
			Q = D^{-1/2} K D^{-1/2} + \Lambda +o(1),\qquad D = \begin{bmatrix}
				\sigma_2(\bw^1)&0\\0&\sigma_2(\bw^2)
			\end{bmatrix}
		\end{equation} where $K = (\kappa_{ij})\in \R_+^{2\times 2}$ with $\det(K)\neq 0$ and $\Lambda = (\lambda_{ij}) \in \R^{2\times 2}$. Both $K$ and $\Lambda$ are symmetric and if $\kappa_{ij} = 0$ then $\lambda_{ij}\ge 0$.
		Lastly, we assume
		\begin{equation}\label{eqn:weightWindow}
			\sqrt{\frac{\sigma_2(\bw^1)}{\sigma_2(\bw^2)}} = 1 + \alpha c_n^{-1} + o(c_n^{-1}) \qquad\textup{where} \qquad c_n = \frac{1}{\sqrt{\sigma_2(\bw^1)\sigma_2(\bw^2)}}.
		\end{equation}  
	\end{assumption}
	
	\begin{remark}
		The assumption that $\det(K)\neq 0$ is so that we are dealing with a near-critical rank-2 random graph instead of a rank-2 perturbation of a rank-1 random graph.
	\end{remark}
	
	\begin{remark} Note that by scaling either $\bw^1$ or $\bw^2$, there is no loss in generality in assuming that the leading order term on the right-hand side of \eqref{eqn:weightWindow} is $1$. This choice just simplifies some later computations. As we will see below, the assumption \eqref{eqn:weightWindow} corresponds to the critical window identified for the SBM studied by Baslingker et al. in \cite[Section 7]{BBBSW.23}. This also extends the critical window of Wang in \cite{Wang.23}.
	\end{remark}

	Since the connected components of $\G(\bW,Q)$ consist of vertices of both type 1 and type 2 vertices, it is natural to track the number of vertices by type using a vector. We do this as follows. Order the connected components of $\G(\bW,Q)$ as $(\cC_l;l\ge 1)$ in such a way that $\cM_{1,1}\ge \cM_{2,1}\ge\dotsm$ where
	\begin{equation}\label{eqn:massDef}
		\cM_{l,i} = \sum_{(r,i)\in \cC_l} w_r^i
	\end{equation} with ties broken in some measurable, but otherwise arbitrary, way.  Write
	\begin{equation}\label{eqn:BCMdef}
		\bcM_l = \begin{bmatrix}
			\cM_{l,1}\\\cM_{l,2}
		\end{bmatrix}.
	\end{equation}
	The natural space that these elements live is
	\begin{equation*}
		\ell^{2}(\N\to \R_+^2) = \left\{\bz=\left(\begin{bmatrix}
			x_l\\
			y_l
		\end{bmatrix};l\ge 1\right): \sum_{l} (x_l^2+y_l^2)<\infty, x_l,y_l\ge 0\right\}.
	\end{equation*} Note that $\ell^2(\N\to\R_+^2)$ is a cone in the separable Banach space $\ell^2(\N\to \R^2)$ (defined in the obvious way) with norm
	\begin{equation}
		\|\bz-\bz'\|_{2,1}^2 = \sum_{l=1}^\infty \left\|\begin{bmatrix}
			x_l\\
			y_l
		\end{bmatrix}-\begin{bmatrix}
			x_l'\\
			y_l'
		\end{bmatrix} \right\|_1^2 = \sum_{l=1}^\infty (x_l-x_l')^2+(y_l-y_l')^2.\label{eqn:ell21def}
	\end{equation} This functional analysis is only necessary for the following lemma, whose proof is elementary and is omitted.
	\begin{lemma}
		The space $\ell^2(\N\to \R_+^2)$ is Polish with distance $d(\bz,\bz') = \|\bx-\bx'\|_{2,1}$ defined in \eqref{eqn:ell21def}.
	\end{lemma}
	
	We now turn to three scaling regimes. We call these the ``classic'' regime, the ``interacting'' regime and the ``nearly bipartite'' regime.

	\subsection{Classic regime: $\kappa_{ij}>0$}

	\begin{regime}[Classic Regime]
		Assumption \ref{ass:main} holds with $\sum_{i} (\beta_i,\btheta_i)\in \cI$. Moreover, 
		\begin{equation}\tag{C}\label{regime:classic}
			K \in (0,\infty)^{2\times 2} \qquad \textup{ and}\qquad \Lambda\in \R^{2\times 2}
		\end{equation}
		and $K$ is such that the PF eigenvalue is $1$.
	\end{regime}
	
	\begin{remark}
		The assumption on the PF eigenvalue being 1 is essentially the criticality assumption from \cite{BJR.07}.
	\end{remark}
	
	Write $\bu^\creg= (u_1^\creg, u_2^\creg)^T$ as the right PF eigenvector of $K$ normalized so that $\langle \bu^\creg, \boldsymbol{1}\rangle = u_1^\creg + u_2^\creg = 1$. 
	Let us define
	\begin{align}\label{eqn:Zcregdef}
		Z^\creg (t) = W^{\beta^\creg, \btheta^\creg,\lambda^{\creg}}(t),
	\end{align}
	where
	\begin{subequations}
		\begin{align}
			\label{eqn:betacreg} \beta^\creg &= (u_1^\creg)^3 \beta_1 + (u_2^\creg)^3 \beta_2\\
			\label{eqn:bthetacreg}\btheta^\creg &=  u_1^\creg\btheta_1 \bowtie u_2^\creg\btheta_2\\
			\label{eqn:lambdacreg} \lambda^\creg&= \langle \bu^\creg, \Lambda\bu^\creg\rangle.
		\end{align} 
	\end{subequations}

	\begin{theorem}\label{thm:classic} Suppose $\bW$ and $Q$ are as in Regime \ref{regime:classic}. Then in $\ell^2(\N\to\R^2_+)$
		\begin{equation*}
			\left(\bcM_l;l\ge 1\right) \weakarrow \left( \zeta_l^\creg \bu^\creg;l\ge 1 \right)
		\end{equation*}
		where $(\zeta_l^\creg;l\ge 1) = \LL^\downarrow_\infty(Z^\creg)$. 
	\end{theorem}

	\subsection{Interacting regime: $\kappa_{12} = 0$}
	
	In \cite{KL.21}, Konarovskyi and Limic identified a new critical regime for the SBM with $m$ types consisting of $n$ vertices of each type. In the two-type case, they considered a random graph $G^{\SBM}(n,n;P)$ with $n$ vertices of each type $i\in[2]$ and vertex set $[n]\times [2]$ with 
	\begin{equation*}
		\PR((l,i)\sim (r,j)) = p_{ij} =\begin{cases}
			1-\exp(-(n^{-1}+\lambda n^{-4/3}))&:i  =j\\
			1-\exp(-\mu n^{-4/3}) &: i\neq j.
		\end{cases}
	\end{equation*} Reformulating this into our setting, they considered the graph $\G(\bW,Q)$ where
	\begin{equation*}
		\bw^i = n^{-2/3} \mathbf{1}_{[n]}, \qquad Q = 
		n^{1/3} \begin{bmatrix}
			1&0\\0&1
		\end{bmatrix} + \begin{bmatrix}
			\lambda & \mu\\
			\mu & \lambda
		\end{bmatrix}+o(1).
	\end{equation*}
	They prove that there exists a two-parameter family of $\ell^2_\downarrow$-valued random variables $$\bX^{\operatorname{IMC}} = (\bX^{\operatorname{IMC}}(\lambda,\mu);\lambda\in\R,\mu\ge 0)$$ called the \textit{(standard) interacting multiplicative coalescent} such that
	\begin{equation*}
		\operatorname{ORD}(\|\bcM_{l}\|_1 ; l\ge 1) \weakarrow \bX^{\operatorname{ICM}} (\lambda,\mu)\qquad \textup{ in}\qquad \ell^2_\downarrow.
	\end{equation*}
	A more detailed account of the limiting process can be deduced from the results by Konarovskyi, Limic and the author in the forthcoming \cite{CKL.24+}; however, here we present a slightly different formulation.
	
	\begin{regime}[Interacting Regime]
		Assumption \ref{ass:main} holds with $ (\beta_i,\btheta_i)\in \cI$ for each $i$. Moreover,
		\begin{equation}\tag{I}\label{regime:interacting}
			K = \begin{bmatrix}
				1&0\\
				0&1
			\end{bmatrix} \qquad \textup{ and}\qquad \Lambda = \begin{bmatrix}
				\lambda_{11} & \lambda_{12}\\
				\lambda_{12}& \lambda_{22}
			\end{bmatrix} 
		\end{equation}
		where $\lambda_{12}>0$.
	\end{regime}
	\begin{remark}
		The assumption that $\kappa_{ii} = 1$ is simply a criticality assumption for the naturally embedded subgraphs $\G(\bw^i,q_{ii}) \hookrightarrow \G(\bW,Q)$ consisting of just the type $i$ vertices. The assumption that $\lambda_{12}>0$ is solely so that we do not have two independent random graphs.
	\end{remark}
	
	Let $Z_i = (Z_i(t);t\ge 0)$ be two independent thinned L\'{e}vy processes defined by 
	\begin{equation}\label{eqn:Zidef}
		Z_i(t) = W_i^{\beta_i,\btheta_i,\lambda_{ii}}(t).
	\end{equation}
	Let
	\begin{equation}\label{eqn:ZiIregDef}
		Z^\ireg(t) = Z_1(t) + \lambda_{12} \inf\{u: Z_2(u)<-\lambda_{12} t\}.
	\end{equation}
	
	\begin{theorem}\label{thm:interact}
		Suppose $\bW$ and $Q$ are as in Regime \ref{regime:interacting}. Then in $\ell^2_\downarrow$, 
		\begin{equation*}
			\left(\cM_{l,1};l\ge 1\right) \weakarrow \left(\zeta^\ireg_l;l\ge 1\right)
		\end{equation*}
		where $(\zeta_l^\ireg;l\ge 1) = \LL_\infty^\downarrow(Z^\ireg).$
	\end{theorem}

	\subsection{Nearly bipartite regime: $\kappa_{ii} = 0$}

	\begin{regime}[Nearly Bipartite Regime] Assumption \ref{ass:main} holds with $\sum_{i} (\beta_i,\btheta_i)\in \cI.$ Moreover,
		\begin{equation}\tag{BP}\label{regime:bipartite}
			K = \begin{bmatrix}
				0&1\\
				1&0
			\end{bmatrix} \qquad \textup{ and }\qquad \Lambda = \begin{bmatrix}
				\lambda_{11} & \lambda_{12}\\
				\lambda_{12}& \lambda_{22}
			\end{bmatrix}\textup{  with  }\lambda_{ii}\ge 0.
		\end{equation}
	\end{regime}

	Set $\btheta = \btheta_1\bowtie \btheta_2$ and $\beta = \beta_1+\beta_2$ and define
	\begin{equation}\label{eqn:zbipdef}
		Z^\bip(t) = W^{\beta,\btheta,0}(t) + (\lambda_{11}+2\lambda_{12} + \lambda_{22})t.
	\end{equation}

	The main theorem for the nearly-bipartite regime is the following.
	\begin{theorem}
		Suppose $\bW$ and $Q$ are as in Regime \ref{regime:bipartite}. Then in $\ell^2(\N\to \R_+^2)$
		\begin{equation}\label{thm:bipartite}
			\left(\bcM_{l};l\ge 1\right)\weakarrow \left(\zeta_l^\bip \boldsymbol{1};l\ge 1\right)
		\end{equation} where $(\zeta^\bip_l;l\ge 1) = \LL_\infty^\downarrow (Z^\bip)$ and $\boldsymbol{1} = (1,1)^T.$
	\end{theorem}
	
	Observe the eigenvector of $K$ with eigenvalue $1$ is $\bu^\bip = (1,1)^T$ and
	\begin{equation*}
		\lambda_{11}+2\lambda_{12} + \lambda_{22} = \langle \bu, \Lambda \bu\rangle.
	\end{equation*}
	So, modulo a trivial scaling (formalized in Corollary \ref{cor:vectorvalue} below), na\"ively taking the limit of Theorem \ref{thm:classic} as $\kappa_{ii}\downarrow 0$ yields Theorem \ref{thm:bipartite}. It would be interesting if this can be formalized into a precise theorem.
	
	\section{Connections with Related Results}\label{sec:discussion}

	The assumptions that we place on the model in Regime \ref{regime:classic} are fairly different than those appearing in \cite{BBBSW.23}. In Regime \ref{regime:bipartite}, the conditions we analyze are much more general than those in \cite{Federico.19} or \cite{Wang.23}. So, before turning to the proof of main results, let us mention how our results can recover and extend certain aspects of these results. The reader can safely skip this section.
	
	\subsection{Classic Regime}

	As mentioned above, Miermont \cite{Miermont.08} shows that under certain conditions including a finite-variance assumption, the height process of a forest of critical multitype Bienaym\'{e}-Galton-Watson (BGW) trees converges after appropriate rescaling to a reflected Brownian motion. His proof relies on a ``reduction of types'' argument that allows him to go from $K$ type BGW forest to a $(K-1)$ type BGW forest. This method allows for explicit representations of the scaling constants in terms of the moments of the offspring distribution. This ``reduction of types'' method was extended to the $\alpha$-stable case in \cite{BerzunzaOjeda.18}. See also the infinitely many types work of \cite{deRaphelis.17}. For an associated branching process scaling limit, one can see the classical result of Joffe and M\'{e}tivier \cite{JM.86} in the critical regime. 
	
	The results of Miermont \cite{Miermont.08} involve a closely related multi-type branching process; however, it does not appear that his approach easily extends to the random graph case. Indeed, Miermont performs the reduction of types in order to obtain useful representations of probability generating functions for the offspring distribution of a single-type BGW forest. This does not appear possible in the random graph context and so we work on the level of exploration processes instead.
	
	For results in the random graph context we know of only two relevant references in this classical regime. The first is the general result of Bollob\'as, Janson and Riordan \cite{BJR.07} which identifies the critical point for a very general model of random graphs. Their results do not extend to establishing the the precise asymptotic size of the largest connected component in the critical case; however, we do mention in the finite-type case Theorem 9.10 therein describes the asymptotic proportion of vertices of each type in the giant component of the supercritical random graph in terms of an eigenvector of the kernel. The appearance of the Perron-Frobenius (PF) eigenvector $\bu$ in Theorem \ref{thm:classic} can be viewed in this light. We also note that this ``convergence-of-types'' result is a well-known result in the study of super-critical branching processes dating back to Kesten and Stigum \cite{KS.66}. See also \cite{KLPP.97} as well as \cite{Miermont.08,deRaphelis.17}.
	
	For results at criticality we refer to \cite{BBBSW.23}. Therein, the authors show that certain models of critical random graphs under failry general assumptions lie in the ``basin of attraction'' of the Erd\H{o}s-R\'{e}nyi random graph. A consequence of their general result is the following simplification of Theorem 7.2 therein. Let  $G^{\SBM}(n_1, n_2; P)$ be the SBM consisting of $n_i$ many vertices of type $i$ where a type $i$ vertex is connected to a type $j$ vertex independently with probability $p_{ij}$.
	\begin{theorem}[Baslingker et al. {\cite[Theorem 7.2]{BBBSW.23}}]\label{thm:BBSW}
		Suppose as $n = n_1+n_2\to\infty$ that
		\begin{enumerate}
			\item $n_i = \mu_i n + b_i n^{2/3} + o(n^{2/3})$ where $\mu_i>0$, $b_i\in\R$ and $\mu_{1}+\mu_2 = 1$;
			\item $p_{ij} = n^{-1} k_{ij}+ n^{-4/3} a_{ij} + o(n^{-4/3})$ is symmetric and $k_{ij}>0$ and $a_{ij}\in \R$ for each $i,j\in[2]$;
			\item The PF root of $\widetilde{M}:=\begin{bmatrix}
				k_{11}\mu_{1} &k_{12}\mu_2\\ k_{21}\mu_1 & k_{22}\mu_{2}
			\end{bmatrix}$ is $1$.
		\end{enumerate}
		Let $\widetilde{\bv},\widetilde{\bu}$ be the left- and right- eigenvecotrs of $M$ (respectively) normalized so that $\langle \widetilde{\bu}, \boldsymbol{1}\rangle = \langle \widetilde{\bv}, \widetilde{\bu}\rangle = 1$. 
		Then in the product topology
		\begin{equation*}
			\left(n^{-2/3}\#\cC_l;l\ge 1\right) \weakarrow \chi^{1/3} \bzeta^{1,\bzer, \chi^{2/3} \lambda}
		\end{equation*} where for $\widetilde{B} = \operatorname{diag}(b_1,b_2)$, $\widetilde{D}= \operatorname{diag}(\mu_1,\mu_2)$, $\widetilde{A} = (a_{ij})$ and $\widetilde{K} = (k_{ij})$
		\begin{equation*}
			\chi = \frac{v_1u_1^2 + v_2 u_2^2}{(v_1+v_2)(\mu_1 u_1+ \mu_2 u_2)^2} \qquad \lambda = \frac{\langle \widetilde{\bv}, (\widetilde{A}\widetilde{D} + \widetilde{K}\widetilde{B})\widetilde{\bu}\rangle }{(v_1+v_2)(\mu_1u_1+\mu_2 u_2)}
		\end{equation*}where $(\cC_l;l\ge 0)$ are the connected components of $G^{\SBM}(n_1,n_2,P)$ ordered from largest to smallest.
	\end{theorem}
	
	The assumptions above can be converted into our assumptions by setting $\bw^i =  \mu_{i}^{-1/2} n^{-2/3}\boldsymbol{1}_{n_i}
	$ where $\boldsymbol{1}_{m}$ is the length $m$ vector of all ones. 
	Note that
	\begin{align*}
		\sigma_{2}(\bw^i) = n^{-1/3}+ \frac{b_i}{\mu_i} n^{-2/3}+o(n^{-2/3})\qquad 
		\sigma_3(\bw^i) = \mu_{i}^{-1/2} n^{-1} + o(n^{-1}).
	\end{align*}
	Hence, for any $i,j\in[2]$ \begin{equation*}
		(\sigma_2(\bw^i)\sigma_2(\bw^j))^{-1/2} = n^{1/3} - \frac{1}{2}\left(\frac{b_i}{\mu_i}+\frac{b_j}{\mu_j}\right) + o(1).
	\end{equation*}
	and $c_n = \frac{1}{\sqrt{{\sigma_2(\bw^1)\sigma_2(\bw^2)}}} = (1+o(1))n^{1/3}$. Therefore,  
	\begin{equation*}
		\bw^i\rightsquigarrow (\mu^{-1/2}_i, \bzer)
		\qquad\textup{and}\qquad
		\sqrt{\frac{\sigma_2(\bw^1)}{\sigma_2(\bw^2)}} = 1+ \left(\frac{b_1}{\mu_1}-\frac{b_2}{\mu_2}\right) c_n^{-1} + o(c_n^{-1}).
	\end{equation*}

	Looking at the connection probabilities, and using the approximation $1-e^{-x} = x+O(x^2)$ as $x\to 0$, we can find the appropriate matrix $Q$ in Assumption \ref{ass:main} by solving
	\begin{align*}
		q_{ij} \mu_i^{-1/2}n^{-2/3} \mu_j^{-1/2} n^{-2/3} = p_{ij}.
	\end{align*}
	Therefore, using condition 2 in the above theorem
	\begin{align*}
		q_{ij} &= \sqrt{\mu_{i}\mu_j} n^{4/3} (n^{-1} k_{ij}+n^{-4/3} a_{ij}) = \sqrt{\mu_i} k_{ij} \sqrt{\mu_j}  n^{1/3} + \sqrt{\mu_i} a_{ij} \sqrt{\mu_j}\\
		&= \frac{\sqrt{\mu_i}k_{ij}\sqrt{\mu_j}}{\sqrt{\sigma_2(\bw^i)\sigma_2(\bw^j)}} + \sqrt{\mu_{i}}a_{ij}\sqrt{\mu_j} + \frac{1}{2} \frac{b_i}{\sqrt{\mu_{i}}} k_{ij} \sqrt{\mu_j} + \frac{1}{2} \sqrt{\mu_i} k_{ij} \frac{b_j}{\sqrt{\mu_j}} + o(1).
	\end{align*}
	Hence, we obtain that $Q$ in Assumption \ref{ass:main} is of the form
	\begin{equation*}
		Q = \begin{bmatrix}
			\sigma_2(\bw^1)^{-1/2}&0\\
			0& \sigma_2(\bw^2)^{-1/2}
		\end{bmatrix} K \begin{bmatrix}
			\sigma_2(\bw^1)^{-1/2}&0\\
			0& \sigma_2(\bw^2)^{-1/2}
		\end{bmatrix} + \Lambda
	\end{equation*} where (recalling that matrices in Theorem \ref{thm:BBSW})
	\begin{equation*}
		K =\widetilde{D}^{1/2}\widetilde{K} \widetilde{D}^{1/2} \qquad \Lambda = \widetilde{D}^{1/2}\widetilde{A} \widetilde{D}^{1/2} + \frac{1}{2} \widetilde{D}^{-1/2} \widetilde{B} \widetilde{K}\widetilde{D}^{1/2}  + \frac{1}{2} \widetilde{D}^{1/2} \widetilde{K} \widetilde{B}\widetilde{D}^{-1/2} 
	\end{equation*}
	Observe that since $\widetilde{\bu}$ is an eigenvector of $\widetilde{M} = \widetilde{K} \widetilde{D}$ with eigenvalue $1$ then $\bu = \widetilde{D}^{1/2} \widetilde{\bu}$ is an eigenvector of $K$ with eigenvalue $1$. Similarly, $\bv = \widetilde{\bv} \widetilde{D}^{-1/2}$ is a left-eigenvector of $K$ with eigenvalue 1. Also note that since $K$ is symmetric $ \bv = c \bu^T$ for some constant $c$. Hence, for any matrix $C\in \R^{2\times 2}$ we can symmetrize
	\begin{equation*}
		\langle \bv, C \bu\rangle = \frac{1}{2} \langle \bv, (C+C^T) \bu\rangle,
	\end{equation*}
	and so 
	\begin{align*}
		\langle& \widetilde{\bv}, \widetilde{A}\widetilde{D}\widetilde{\bu}\rangle = c \langle \bu, \widetilde{D}^{1/2} \widetilde{A} \widetilde{D}^{1/2} \bu\rangle\\
		\langle &\widetilde{\bv} , \widetilde{K} \widetilde{B} \widetilde{\bu}\rangle = \langle \bv, \widetilde{D}^{-1/2}\widetilde{K} \widetilde{B}\widetilde{D}^{1/2} \bu\rangle\\
		&\qquad\qquad\qquad = \frac{c}{2} \langle \bu, \frac{1}{2}\left( \widetilde{D}^{-1/2} \widetilde{B} \widetilde{K}\widetilde{D}^{1/2}  +  \widetilde{D}^{1/2} \widetilde{K} \widetilde{B}\widetilde{D}^{-1/2}  \right)\bu\rangle.
	\end{align*}
	Said another way 
	\begin{equation*}
		\langle \widetilde{\bv} ,(\widetilde{A}\widetilde{D}+\widetilde{K}\widetilde{B}) \widetilde{\bu}\rangle = c \langle \bu, \Lambda\bu\rangle,
	\end{equation*}
	hence, up to scaling, Theorem \ref{thm:classic} agrees with those of \cite{BBBSW.23}.

	\subsection{Bipartite Regime}
	
	In this section we will restrict our attention solely to case to the purely bipartite regime where $\lambda_{ii} = 0$ for each $i\in[2]$. 
	
	Recently, much interest has been devoted to studying bipartite random graphs because of their connection with random intersection graphs. See \cite{vdHKV.21,vdHKV.22,BST.14,Wang.23,Federico.19,DK.09,Behrisch.07,Bloznelis.17,Bloznelis.10,Bloznelis.13,FSSC.00} for example. Up to now, most of the work on bipartite graphs have focused on the homogeneous settings where vertices have roughly the same degree; however, Theorem \ref{thm:bipartite} sheds new light on the inhomogeneous setting.

	The simplest model of a random bipartite graph is the bipartite Erd\H{o}s-R\'{e}nyi random graph $B(n,m,p)$ with $n$ many vertices on one side (say the left) and $m$ many vertices on the other (say the right) with each edge between the two parts appearing with probability $p$. In \cite{BST.14, Behrisch.07}, it was shown that if $m = \lfloor n^\alpha \rfloor$ (for $\alpha\neq 1)$ and $p = p(n) = \frac{c}{\sqrt{nm}}$ then there is a phase transition in the maximum number of left vertices in the connected components and this phase transition occurs at the critical value $c = 1$. The phase transition also occurs at $c = 1$ in the case where $m = n$ (i.e. $\alpha = 1$) using the general results of Bollob\'as, Janson and Riordan \cite{BJR.07}. The component sizes at the critical value were further studied by Federico \cite{Federico.19} for $\alpha\neq 1$ and by De Ambroggio \cite{DeAmbroggio.22} for $\alpha =1$. We summarize all of these results in the following table. Let $L_1^{(n)}$ denote the maximum number of left vertices in any connected component of $B(n,m,p)$. Then 
	\begin{center}
		$L^{(n)}_1= $   \begin{tabular}{|c|c|c|c|}\hline
			& $c<1 $   &  $c = 1$ & $c>1$ \\ \hline
			$  \alpha<1 $ &$O_\PR\left(\sqrt{n^{1-\alpha}} \log n\right) ${\cite{Behrisch.07}}   & $\Theta_\PR\left(\sqrt{n^{1+\alpha/3}}\right)$\cite{Federico.19} &$ \Theta_\PR(\sqrt{n^{1+\alpha}}) $\cite{Behrisch.07} \\\hline
			$ \alpha  = 1$& $O_\PR(\log n)$\cite{BJR.07}& $\Theta_\PR(n^{2/3})$\cite{DeAmbroggio.22}& $\Theta_\PR(n) $ \cite{BJR.07}\\\hline
			$ \alpha >1 $ & $O_{\PR}(\log n) ${\cite{Behrisch.07}}& $\Theta_\PR(n^{2/3})$ \cite{Federico.19} &$ \Theta_\PR(n) $\cite{Behrisch.07} \\\hline
		\end{tabular}
	\end{center}
	We should mention that the above table does not properly convey the strength of the results in the referenced works but is instead meant to convey briefly the various results in the same setting. 
	
	Deeper understanding of the critical regime was obtained by Wang \cite{Wang.23}. The limits she construct depend on three different \textit{clustering regimes} described as follows.
	\begin{enumerate}
		\item[(\namedlabel{LightWang}{L})] $m/n\to+\infty$ and $p = \frac{1}{\sqrt{nm}}(1+\lambda n^{-1/3})$
		\item[(\namedlabel{ModerateWang}{M})]$m/n = \theta+o(n^{-1/3})$ and $p = \frac{1}{\sqrt{nm}}(1+\lambda n^{-1/3})$ for some $\theta>0$
		\item[(\namedlabel{HeavyWang}{H})]$m/n\to0$ and $p = \frac{1}{\sqrt{nm}}(1+\lambda m^{-1/3})$
	\end{enumerate} We mention that Federico \cite{Federico.19} does study the \eqref{LightWang} and \eqref{HeavyWang} regimes in the restricted cases where $m = n^\alpha$.
	
	The following theorem, as stated, is a consequence of Propositions 2.7, 2.8 and Corollary 2.9 in \cite{Wang.23} along with Lemma \ref{lem:summingtwothinned} below. To state the result we let $\bS_L(\cC)$ denote the number of left vertices in a connected component, $\cC$, of $B(n,m,p)$ and we let $\bS_R(\cC)$ be the number of right vertices in the same connected component.
	\begin{theorem}[Wang \cite{Wang.23} and Federico \cite{Federico.19}]\label{thm:WF} Let $(\cC_j^{(n)};j\ge 1)$ denote the connected components of $B(n,m,p)$ listed in arbitrary order. The following limits hold in the product topology on $\ell^2$.
		\begin{enumerate}
			\item In regime \eqref{LightWang} it holds that
			\begin{align*}
				&\ORD\left(n^{-2/3}\bS_L(\cC_j^{(n)});j\ge 1\right) \weakarrow \bzeta^{1,0,2\lambda}\\
				&\ORD\left(n^{-1/6}m^{-1/2}\bS_R(\cC_j^{(n)});j\ge 1\right)\weakarrow \bzeta^{1,0,2\lambda}.
			\end{align*}
			\item In regime \eqref{ModerateWang} it holds that
			\begin{align*}
				&\ORD\left(n^{-2/3}\bS_L(\cC_j^{(n)});j\ge 1\right) \weakarrow \LL^\downarrow_\infty \bzeta^{1+\theta^{-1/2}, \bzer, \lambda}\\
				&\ORD\left(n^{-2/3}\bS_R(\cC_j^{(n)});j\ge 1\right) \weakarrow \sqrt{\theta}\bzeta^{1+\theta^{-1/2}, \bzer, \lambda}
			\end{align*}
			for two independent thinned L\'{e}vy processes $W^{1,0,\lambda}$ and $\widetilde{W}^{\theta^{-1/2},0,\lambda}$.
			\item In regime \eqref{HeavyWang} it holds that
			\begin{align*}
				&\ORD\left(n^{-1/6}m^{-1/2}\bS_L(\cC_j^{(n)});j\ge 1\right) \weakarrow \bzeta^{1,0,2\lambda},\\
				&\ORD\left(n^{-2/3}\bS_R(\cC_j^{(n)});j\ge 1\right)\weakarrow \bzeta^{1,0,2\lambda}.
			\end{align*}
		\end{enumerate}
		In each limiting regime the convergences hold jointly with respect to the same limiting objects. 
	\end{theorem}
	
	\begin{remark}
		We state the above result in terms of the product topology; however, some can be easily extended to the $\ell^2$ topology on $\ell^2_\downarrow$ using the theory of size-biased point processes from \cite{Aldous.97}.
	\end{remark}
	
	Theorem \ref{thm:bipartite} can be used to recover the above results. Indeed, for the light regime \eqref{LightWang} simply take 
	$\bw^{1} = n^{-2/3} \mathbf{1}_{n}$ and $  \bw^2 = n^{-1/6} m^{-1/2}\mathbf{1}_m$
	and note that
	\begin{equation*}
		\sigma_2(\bw^{1}) = \sigma_2(\bw^2) = n^{-1/3},\qquad \frac{\sigma_3(\bw^2)}{\sigma_2(\bw^2)^3} = \sqrt{\frac{n}{m}} = o(1)\qquad \frac{\sigma_3(\bw^1)}{\sigma_2(\bw^1)^3} = 1.
	\end{equation*} Hence, $\bw^1 \rightsquigarrow (1,\bzer)$ and $\bw^2 \rightsquigarrow (0,\bzer)$ and so
	\begin{equation*}
		Z^\bip(t) = W^{1,\bzer,2\lambda_{12}}(t).
	\end{equation*} As
	\begin{equation*}
		q_{12} = \frac{1}{\sqrt{\sigma_2(\bw^1)\sigma_2(\bw^2)}} + \lambda_{12}  = n^{1/3}+\lambda_{12}
	\end{equation*}
	we have
	\begin{equation*}
		\PR((l,i)\sim (r,j)) = 1-\exp\left(- n^{-\frac{2}{3}} n^{-\frac{1}{6}}m^{-\frac{1}{2}} (n^{\frac{1}{3}} +\lambda_{12})\right) = \frac{1+ \lambda_{12} n^{-1/3} +o(n^{-1/3}) }{\sqrt{nm}}.
	\end{equation*} Hence, our results recover the light regime \eqref{LightWang} described in Theorem \ref{thm:WF} above. The heavy regime \eqref{HeavyWang} follows from interchanging the roles of $\bw^1$ and $\bw^2$. The moderate regime \eqref{ModerateWang} follows by taking the same weight vectors but instead noting that $\bw^2 \rightsquigarrow (\sqrt{\theta},\bzer)$ in that case.

	\section{Exploring the graph}\label{sec:exploration}

	In this section we recall from \cite{CKL.22} the encoding of the connected components of a two-type DCSBM. In \cite{CKL.22}, the authors show that we can encode the random graph $\G(\bW,Q)$ using a random field $\bbX = (\bbX(s,t);s,t\ge 0)$ with values in $\R^2$ and the first hitting times of Chaumont and Marolleau \cite{Chaumont:2020, Chaumont:2021} which extends the discrete analysis of Chaumont and Liu from \cite{Chaumont:2016}.

	\subsection{First hitting times of fields}
	
	As this will play a fundamental role in our scaling limits in the sequel, we spend some time on the construction of first hitting times from \cite{Chaumont:2020}. As we will only focus on the two-type SBMs in this article, we will only focus exclusively on the case where the additive field $F:\R_+^2\to \R^2$.
	
	The field $F$ is defined using four c\`adl\`ag functions $f^{j,i}:\R_+\to \R$ for $i,j\in\{1,2\}$ where $f^{j,i}(0) = 0$ for all $i,j$, $f^{i,i}(t)\ge f^{i,i}(t-)$ and $f^{j,i}$ are non-decreasing for $i\neq j$. More precisely for $\vec{t} = (t_1,t_2)$
	\begin{equation*}
		F(\vec{t}) = \Big(F^1(\vec{t}),F^2(\vec{t})\Big) \qquad F^j(\vec{t}) = \sum_{i=1}^2 f^{j,i}(t_i).
	\end{equation*}Given a vector $\vec{r}\in \R^2_+$ we are interested in solutions $\vec{t}\in [0,\infty]^2$
	\begin{equation}\label{eqn:fieldMin}\tag{$\vec{r},F$}
		F^{j}(\vec{t}-) = f^{j,1}(t_1-) + f^{j,2}(t_2-) = -r_j \textup{ for all }j\textup{ s.t. }t_j<\infty.
	\end{equation} Here we interpret $f^{i,i}(\infty)$ as undefined and, for $i\neq j$, $f^{j,i}(\infty) = f^{j,i}(\infty-) = \lim_{t\to\infty} f^{j,i}(t)\in[0,\infty]$. 
	Observe that \eqref{eqn:fieldMin} is vacuously true for $\vec{t} = (\infty,\infty)$ and so $f^{i,i}(\infty)$ being undefined posses no issue.
	
	The following is a key lemma in \cite{Chaumont:2020}. 
	\begin{lemma}[Existence of first hitting times \cite{Chaumont:2020}] Suppose that $F$ is as above and $\vec{r}\in \R^2_+$. Then there exists a unique solution $\mathbf{T} =\mathbf{T}(\vec{r}) = \mathbf{T}(\vec{r},F) \in[0,\infty]^2$ such that for all other solutions $\vec{t}$ to \eqref{eqn:fieldMin} satisfies $\mathbf{T}\le \vec{t}.$ That is $\mathbf{T} = (T_1,T_2)$ satisfies $T_i\le t_i$ for $i = 1,2$.
	\end{lemma}
	
	We include the following construction of the solution $\mathbf{T}$ from \cite{Chaumont:2020} as this will be vital to subsequent analysis. Define $\vec{u}^{(0)} = \vec{0}$ and for each $n\ge 1$ define $\vec{u}^{(n)}$ by
	\begin{equation}\label{eqn:uconstr1}
		u_j^{(n)} = \inf\left\{t: f^{j,j}(t-) + f^{j,i} (u_i^{(n-1)}-) = -r_j \right\}.
	\end{equation}
	That is, if $\tau^j(v) = \inf\{t: f^{j,j}(t-) = - v\}$ then
	\begin{equation}\label{eqn:uandlj}
		u^{(n)}_j = \tau^j\left(r_j+ f^{j,i} (u^{(n-1)}_i-)\right)\qquad \textup{for  }  i\neq j.
	\end{equation}
	\begin{lemma}[Construction of $\mathbf{T}$ \cite{Chaumont:2020}]\label{lem:construct}
		Using the above notation, it holds that $u^{(n)}_j\le u^{(n+1)}_j$ for all $n\ge 0$ and $j = 1,2$. Moreover, $        \mathbf{T}(\vec{r}) = \lim_{n\to\infty} \vec{u}^{(n)}.
		$    \end{lemma}
	
	Before turning to our special case, we recall the following lemma that was used heavily in \cite{CKL.22}. 
	\begin{lemma}\label{lem:leftCont}
		For any vector $\vec{r}$ the map $s\mapsto \mathbf{T}(s\vec{r},F)$ is left-continuous. 
	\end{lemma}
	
	We know turn to a special case where $\vec{r}= (r,0)$ for some $r>0$. To state this we will write $g_-(t) = g(t-)$.
	
	\begin{lemma}\label{lem:singleProcessRep}
		Suppose that $\vec{r} = (r,0)$ for some $r>0$. Then $\mathbf{T}(\vec{r}) = (T_1(r),T_2(r))$ satisfies
		\begin{equation}\label{eqn:T1(r)}
			T_1(r) = \inf\left\{t: f^{1,1}_-(t)  + f^{1,2}_-\circ\tau^2\circ  f^{2,1}_-(t) = -r \right\}
		\end{equation}
		and
		\begin{equation}\label{eqn:T2(r)}
			T_2(r) = \tau^2(f^{2,1}_-(T_1(r))).
		\end{equation}
	\end{lemma}
	\begin{proof}
		We have by the construction in Lemma \ref{lem:construct} that $\bT(\vec{r}) = \lim_{n\to\infty} \vec{u}^{(n)}$. 
		
		Since $r_2 = 0$, equation \eqref{eqn:uandlj} is equivalent to 
		\begin{equation*}
			u_2^{(n)} = \tau^{2} \big(f^{2,1}_-(u_1^{(n-1)})\big)
		\end{equation*}
		and so \eqref{eqn:T2(r)} follows from the observation that for $i\neq j$ both $f_{-}^{j,i}$ and $\tau^j$ are left-continuous functions.
		
		To show that \eqref{eqn:T1(r)} holds, we use \eqref{eqn:uconstr1} to get
		\begin{equation*}
			u_1^{(n+1)} = \inf\left\{t: f^{1,1}_-(t) + f^{1,2}_-\left(\tau^2\circ f_{-}^{2,1} (u_1^{(n-1)})\right)  = - r\right\}.
		\end{equation*}
		Taking the limit as $n\to\infty$, and using the left-continuity of all functions involved in the RHS above, we get
		\begin{equation}\label{eqn:ubound}
			T_1(r) = \inf \{t: f^{1,1}_-(t) + f_-^{1,2}\circ\tau^2\circ f_-^{2,1} (T_1(r)) = -r\}.
		\end{equation}
		Also, note that $u^{(n)}_1\le u_1^{(n+1)}\le T_1(r)$, and $f^{1,2}, f^{2,1}$ and $\tau^2$ are non-decreasing functions. A simple induction argument yields
		\begin{equation*}
			u_1^{(n+1)} \le \inf\{t: f^{1,1}_-(t) + f_{-}^{1,2}\circ \tau^2\circ f_-^{2,1}(t) = -r\}.
		\end{equation*} 
		Taking the limit as $n\to\infty$ above we get
		\begin{equation*}
			T_1(r) \le \inf\{t: f^{1,1}_-(t) + f_-^{1,2}\circ \tau^2 \circ f^{2,1}_-(t) = -r\}.
		\end{equation*} 
		We also have from \eqref{eqn:ubound} that 
		\begin{equation*}
			f_{-}^{1,1}(T_1(r)) + f^{1,2}_-\circ\tau^2\circ f_-^{2,1}(T_1(r)) = -r.
		\end{equation*}Combining the last two displayed equations we see that \eqref{eqn:T1(r)} holds. 
	\end{proof}

	\subsection{Fields and Random Graphs}
	
	Now that we have the construction of the first hitting times for additive fields, we can turn to one of the key constructions in \cite{CKL.22}. To begin, we define the stochastic processes that will play a key role in our subsequent analysis. Following \cite{Limic.19,Limic.19_Supplement}, we introduce
	\begin{align}\label{eqn:Nidef}
		N^i(t) &= \sum_{l=1}^{\len(\bw^i)} w_l^i 1_{[\xi_l^{i} \le q_{ii}t]}
	\end{align}
	where $\xi_l^i\sim \Exp(w_l^i)$ independent exponential random variables of rate $w_l^i$ (and mean $1/w_l^i$). As in \cite{CKL.22} we write
	\begin{align}
		\label{eqn:Xdef}X^{1,1}(t) &= N^1(t)  -t  &X^{1,2}(t)&= \frac{q_{1,2}}{q_{1,1}} N^2(t)\\  X^{2,1}(t)&= \frac{q_{1,2}}{q_{2,2}} N^{1}(t)
		& X^{2,2}(t)&= N^2(t).\nonumber
	\end{align}
	Let $\bX$ denote the corresponding additive field
	\begin{equation*}
		\bX(\vec{t}) = \left(\sum_{i=1}^2 X^{1,i}(t_i), \sum_{i=1}^2 X^{2,i}(t_i)\right)
	\end{equation*}
	
	Recall that we label the connected components of $\G(\bW,Q)$ by $(\cC_l;l\ge 1)$ in such a way that $\cM_{l,1}\ge \cM_{l,2}\ge \dotsm$ where $\cM_{l,i}$ are defined as in \eqref{eqn:massDef} and that we define the vectors $\bcM_j$ by \eqref{eqn:BCMdef}. Define the matrix $R$ by 
	\begin{equation*}
		R = \begin{bmatrix}
			1& q_{12}/{q_{11}}\\
			q_{21}/q_{22} &1 
		\end{bmatrix} = \operatorname{diag}(Q)^{-1} Q.
	\end{equation*}

	Let $\vec{\rho}\in \R^2_+\setminus\{0\}$ be fixed. Let $(t_l = t_l(\vec{\rho});l\ge 1)$ be the discontinuity times of $t\mapsto \bT(s\vec{\rho};\bX)$ in the order in which they appear. Denote the respective jumps by $(\bDelta_l = \bDelta_l(\vec{\rho});l\ge 1)$ so that
	\begin{equation}\label{eqn:tjdef}
		\bDelta_l = \lim_{t\downarrow t_l}\bT(\vec{\rho}t) - \bT(\vec{\rho}t_l).
	\end{equation}We also set $\cS_l = \cS_l(\vec{\rho}) = \vec{\rho}^T \operatorname{diag}(Q)^{-1} \bcM_l$. 
	
	We now turn to a size-biasing operation. Given a collection of non-negative scalars $(s_j;j\in[n])$ with $s_j>0$, we say that $\pi$ is a size-biased permutation if for every permutation $\tau\in \fS_n$, the symmetric group on $n$ letters, 
	\begin{equation*}
		\PR(\pi = \tau) = \prod_{l=1}^n \frac{s_{\tau(l)}}{\sum_{k\ge l} s_{\tau(k)}}.
	\end{equation*} We will say that $\pi\sim \operatorname{SB}(s_j).$ We extend this slightly to the case where some $s_j$ are zero. If $(s_j;j\in[n])$ with $s_j\ge 0$ and at least one non-zero, we write $\operatorname{SB}(s_j)$ as a random function $\pi:[\#A] \to A$ where $A = \{j:s_j>0\}$ with law
	\begin{equation*}
		\PR\left(\pi = \tau\right) = \prod_{l=1}^{\#A} \frac{s_{\tau(l)}}{\sum_{k\ge l} s_{\tau(k)}}\qquad \forall \tau:[\#A]\to A.
	\end{equation*}
	Given a finite list $(o_j;j\in [n])$ of objects, and another finite list $(s_j;j\in [n])$ of sizes $s_j\ge 0$, we will say that $(o_{\pi(j)};j\in[\#A])\sim \operatorname{SB}(o_j,s_j)$ if $\pi\sim \operatorname{SB}(s_j).$ If $(o_j)$ and $(s_j)$ are random, then the random function $\pi$ is generated conditionally independently given $(o_j), (s_j)$. We want to emphasize that when some of the $s_j = 0$, the randomized sequence $\operatorname{SB}(o_j,s_j)$ is incomplete. It is well known that the size-biasing operation can be accomplished by using a family of independent exponential random variables of rates $s_j$.
	
	\begin{theorem}[{\cite[Corollary 4.13]{CKL.22}}]\label{thm:ckl1}
		For each fixed $\vec{\rho}\in \R_+^2\setminus\{0\}$ it holds that 
		$\left(\bDelta_l;l\ge 1\right) \overset{d}{=} \left(R \bcM_l^\circ;l\ge 1 \right)
		$ where $(\bcM^\circ_l;l\ge 1)\sim \operatorname{SB}(\bcM_l,\cS_l).$

	\end{theorem}
	\begin{remark}
		Whenever $\vec{\rho}= (1,0)^T$ is the first coordinate vector the resulting size-biased ordering $\cM_j^\circ$ does not include all the connected components. It omits those components $\cC_j$ consisting of only type $2$ vertices.
	\end{remark}
	The above theorem and Lemma \ref{lem:singleProcessRep} lead us to examine the following processes:
	\begin{equation}\label{eqn:LandUdef}
		L^{j}(t) = - \inf_{s\le t} X^{j,j}(s),\qquad U^j(y) = \inf\{t: L^j(t) > y\} ,
	\end{equation}
	and
	\begin{equation}\label{eqn:vdef}
		V(t) = X^{1,1}(t) + X^{1,2}\circ U^2\circ X^{2,1}(t) .
	\end{equation}
	We note in Lemma \ref{lem:singleProcessRep} the functions in \eqref{eqn:T1(r)} and \eqref{eqn:T2(r)} are left-continuous, but here we are using the right-continuous versions. This posses no issue to us due to the following lemma used in \cite{CKL.22}. It is a consequence of Corollary 4.5 found therein.
	\begin{lemma} \label{lem:ascontinuity}For any fixed $\vec{\rho}\in \R_+^2\setminus \{0\}$, almost surely for every $s\ge 0$ the vector $(T_1,T_2)=\lim_{s'\downarrow s} \mathbf{T}(s'\vec{\rho})$ satisfy $N^{1}(T_1) = N^1(T_1-)$ and $N^2(T_2) = N^2(T_2-)$.
	\end{lemma}
	
	Let $(l_j,r_j)$ be the excursion intervals of $V(t)$ above its running minimum:
	\begin{equation}\label{eqn:vexcursionsGen}
		\{t:V(t)>\inf_{s\le t} V(s) \}^\circ = \bigcup_{j}(l_j,r_j)
	\end{equation} in the almost surely unique way that $l_1<r_1<l_2<r_2<\dotsm$. Here we write $A^\circ$ for the interior of the set $A$. 
	Therefore, using Theorem \ref{thm:ckl1} we get the following proposition.
	\begin{proposition}\label{prop:vexcursions}
		Let $\vec{\rho} = (1,0)^T$. Then
		\begin{equation*}
			\left\{      \begin{bmatrix}
				r_j-l_j\\
				U^2\circ X^{2,1}(r_j) - U^2\circ X^{2,1}(l_j-)
			\end{bmatrix}; j\ge 1\right\} \sim \operatorname{SB}(R\bcM_j,\cS_j)
		\end{equation*}
		where $\cS_j = \frac{1}{q_{11}}\cM_{j,1}$ is the scaled mass of the type 1 vertices in $\cC_j$.
		
	\end{proposition}

	\subsection{Bipartite Analog}
	
	There is a technical downside to the above approach. The above results recalled require that $q_{ii}$ be strictly positive for each $i$. Therefore, we cannot directly handle the bipartite case when $q_{ii} = \lambda_{ii} + o(1) = 0$ for some $i$. 
	
	Let us first note the following consequence of \cite{Janson.10a}.
	\begin{lemma}
		Suppose $\bW$ and $Q$ are as in Regime \eqref{regime:bipartite}. Then there exists a sequence $\delta_n\to 0$ sufficiently fast such that $\G(\bW,Q)$ is asymptotically equivalent to $\G(\bW,Q+\delta_nI).$ In particular, we can assume that $q_{ii}>0$ in Regime \ref{regime:bipartite}.
	\end{lemma}

	We will implicitly assume that we are working with an asymptotically equivalent version and write
	\begin{equation}\label{eqn:qremixbp}
		Q = c_n\begin{bmatrix}
			0 & 1\\
			1& 0
		\end{bmatrix} + \Lambda_n\qquad\textup{where}\qquad
		\Lambda_n = \begin{bmatrix}
			\lambda_{11}^{(n)} & \lambda_{12}^{(n)}\\
			\lambda_{12}^{(n)} & \lambda_{22}^{(n)}
		\end{bmatrix}
	\end{equation} where $\lambda_{ii}^{(n)}>0$.

	Let us define $\widetilde{\bw}^i = \eps_i \bw^i$
	for some $\eps_i>0$ which is to be determined. Let $\widetilde{\bW} = (\widetilde{\bw}^1,\widetilde{\bw}^2)$. Observe that if $\widetilde{Q} = (\widetilde{q}_{ji})$ with $q_{ji} = \widetilde{q}_{ji} \eps_i\eps_j$ then
	\begin{equation}\label{eqn:equalgraphs}
		\PR\big((l,i)\sim (r,j)\textup{  in  }\G(\widetilde{\bW},\widetilde{Q})\big) = \PR\big((l,i)\sim (r,j)\textup{  in  }\G({\bW},{Q})\big)\qquad \forall l,r,i,j.
	\end{equation}
	
	For $Q$ as in \eqref{eqn:qremixbp} and any $\eps_1,\eps_2>0$, define
	\begin{align*}
		\widetilde{q}_{ii} &= \eps_i^{-1} (c_n+\lambda_{12}^{(n)}),& \eps_i &=\frac{\lambda_{ii}^{(n)}}{c_n+ \lambda_{12}^{(n)}}, & &\textup{and} &\widetilde{q}_{12} &= \frac{c_n+\lambda_{12}^{(n)}}{\eps_1\eps_2}.
	\end{align*}
	Observe that $\eps_i\eps_j \widetilde{q}_{ji} = q_{ji}$ and so \eqref{eqn:equalgraphs} holds. In the remainder of this subsection, we couple $\G:= \G(\bW,Q)$ with $\widetilde{\G}: = \G(\widetilde{\bW}, \widetilde{Q})$ so that they are equal a.s. as graphs on $\{(l,i):i\in[2],l\in[\len(\bw^i)]\}$.

	Observe that if $\xi_{l}^i\sim \operatorname{Exp}(w_l^i)$ then $\widetilde{\xi}_l^i := \eps_i^{-1} \xi_l^i\sim \operatorname{Exp}(\widetilde{w}_l^i)$. Define 
	\begin{equation*}
		\widetilde{N}^{i}(t):= \sum_{l=1}^{\len(\bw^i)} \eps_i w_l^i 1_{[\widetilde{\xi}_l^i \le \widetilde{q}_{ii}t]} = \eps_i \sum_{l=1}^{\len(\bw^i)} w_l^i 1_{[\xi_l^i\le qt]},\qquad\textup{where  } q = c_n+\lambda_{12}^{(n)}.
	\end{equation*}
	Similar to \eqref{eqn:Xdef}, define
	\begin{align*}
		\widetilde{X}^{1,1}(t) &= -t + \widetilde{N}^{1}(t)  & \widetilde{X}^{1,2}(t) &= \frac{\widetilde{q}_{12}}{\widetilde{q}_{11}} \widetilde{N}^2(t)\\
		\widetilde{X}^{2,1}(t) &= \frac{\widetilde{q}_{12}}{\widetilde{q}_{22}} \widetilde{N}^1(t) & \widetilde{X}^{2,2}(t) &= -t + \widetilde{N}^{2}(t).
	\end{align*} 
	
	Observe that $\widetilde{q}_{12}/\widetilde{q}_{11} = \eps_2^{-1}$ and $\widetilde{q}_{12}/\widetilde{q}_{22} = \eps^{-1}_1$. Hence
	\begin{equation*}
		\widetilde{R} := \operatorname{diag}(\widetilde{Q})^{-1} \widetilde{Q} = \begin{bmatrix}
			1& \eps_2^{-1}\\
			\eps_1^{-1} & 1
		\end{bmatrix} = \begin{bmatrix}
			\eps_1 & 1\\
			1&\eps_2 
		\end{bmatrix}\begin{bmatrix}
			\eps_1^{-1}&0\\
			0&\eps_2^{-1}
		\end{bmatrix}
	\end{equation*} and
	\begin{equation*}
		\frac{\widetilde{q}_{12}}{\widetilde{q}_{jj}} \widetilde{N}^i(t) = \sum_{l=1}^{\len(\bw^i)}w_l^i 1_{[\xi_l^i\le qt]}.
	\end{equation*} It now follows that
	\begin{align}
		\widetilde{X}^{1,1}(t) &=-t+\eps_1\sum_{l=1}^{\len(\bw^1)} w_l^1 1_{[\xi_l^1\le qt]} & \widetilde{X}^{2,1}(t) &= \sum_{l=1}^{\len(\bw^1)} w_l^1 1_{[\xi_l^1\le qt]}\label{eqn:Ydef}\\
		\widetilde{X}^{2,1}(t) &= \sum_{l=1}^{\len(\bw^2)} w_l^2 1_{[\xi_l^2\le qt]} & \widetilde{X}^{2,2}(t) &=-t+\eps_2\sum_{l=1}^{\len(\bw^2)} w_l^2 1_{[\xi_l^2\le qt]}\nonumber.
	\end{align} This reformulation makes proving our limit theorems in the sequel much easier. 
	
	Note that if $\cC$ is a connected component of $\G$ (and hence $\widetilde{\G}$ using the coupling) then
	\begin{equation*}
		\sum_{(l,i)\in \cC} \widetilde{w}_l^i = \eps_i \sum_{(l,i)\in \cC} w_l^i.
	\end{equation*} Hence, if we label the connected components of $\G$ in some arbitrary way by $(\cC_l;l\ge 1)$ these are the same connected components of $\widetilde{\G}$ and, moreover, the associated mass vectors satisfy
	\begin{equation*}
		\widetilde{\bcM}_l = \begin{bmatrix}
			\eps_1&0\\
			0&\eps_2
		\end{bmatrix}\bcM_l.
	\end{equation*}
	Let $R^{\bip} = \begin{bmatrix}
		\eps_1&1\\1&\eps_2
	\end{bmatrix}$
	and note $\widetilde{R} \widetilde{\bcM}_l = R^\bip \bcM_l.$
	Also observe that if $\vec{\rho} = (1,0)^T$ then 
	\begin{equation*}\langle \vec{\rho},\operatorname{diag}(\widetilde{Q})^{-1} \widetilde{\bcM}_l\rangle = \frac{1}{q} \cM_{l,1}\qquad\textup{where} \qquad q = c_n + \lambda_{12}^{(n)} = c_n + \lambda_{12}+o(1).
	\end{equation*} A direct consequence of Proposition \ref{prop:vexcursions} and the above analysis is the following. 
	\begin{proposition}\label{prop:vexcursionsbip} Let $\bW,Q$ be as in Regime \ref{regime:bipartite} subject to \eqref{eqn:qremixbp}. 
		Let
		\begin{equation*}
			V^\bip (t) = \widetilde{X}^{1,1}(t)+\widetilde{X}^{1,2}\circ \widetilde{U}^{2}\circ {\widetilde{X}}^{2,1}(t)
		\end{equation*}
		where
		\begin{equation*}
			\widetilde{U}^j(y) = \inf\{t: \widetilde{X}^{j,j}(t) < -y\}.
		\end{equation*}  Then
		\begin{equation*}
			\left\{   \begin{bmatrix}
				r_j-l_j\\
				\widetilde{U}^2\circ \widetilde{X}^{2,1}(r_j) - \widetilde{U}^2\circ \widetilde{X}^{2,1}(l_j-)
			\end{bmatrix};j\ge 0\right\} \sim \operatorname{SB}(R^\bip \bcM_j,\cS_j)
		\end{equation*}
		where $\cS_j = q^{-1} \cM_{j,1}$ with $q = c_n +\lambda_{12}^{(n)}.$
		
	\end{proposition}

	\section{Some Results in the $M_1$ topology}\label{sec:generaltopology}
	In this section we recall the technical background for the $M_1$ topology which allows us to prove our main results. This will involve recalling various deterministic functionals that are continuous in the $M_1$ topology but not necessarily in the $J_1$ topology. Most of these results we recall from \cite{Whitt.02,Whitt.02a,Whitt.80}. For the sake of space, we only include the functionals in the $M_1$ topology.

	\subsection{Skorohod Topologies} \label{sec:M1}

	Given a c\`adl\`ag function $f:[0,T]\to \R$, let $\Gamma^f =\{(t,y): y\in[f(t-), f(t)]\}$ be the \textit{thin} graph of the function $f$.
	
	\begin{definition}
		Let $\psi_n,\psi_0\in \D([0,T])$.
		\begin{enumerate}
			\item[(J1)] We say that $\psi_n\to \psi_0$ in the $J_1$ topology if there are a sequence of increasing homeomorphisms $\lambda_n:[0,T]\to [0,T]$ such that
			\begin{equation*}
				\sup_{t\le T} \left\{\left|\lambda_n(t) - t\right| + \left|\psi_n\circ\lambda_n(t) -\psi_0(t)\right|\right\} \longrightarrow 0.
			\end{equation*}
			\item[(M1)] We say that $\psi_n\to \psi_0$ in the $M_1$ topology if there exists a continuous functions $t_n,y_n:[0,T]\to \R$ such that $(t_n,y_n):[0,T]\to \Gamma^{\psi_n}$ is onto for all $n\ge 0$ and $t_n$ is non-decreasing such that
			\begin{equation*}
				\sup_{s\in[0,T]} \left\{\left|t_n(s)-t_0(s)\right| + \left|y_n(s) - y_0(s)\right| \right\} \longrightarrow 0.
			\end{equation*}
		\end{enumerate}
		If $\psi_n\to \psi_0$ in topology $\ast\in\{J_1,M_1\}$ then we write $\psi_n\overset{\ast}{\longrightarrow} \psi_0$.
		
		Lastly, we say that $\psi_n\overset{\ast}\longrightarrow \psi_0$ in $\D(\R_+)$ if the same convergence holds in $\D([0,T])$ for every continuity point $T$ of $\psi_0$.
	\end{definition}
	
	The following lemma is useful and is easy to see from the definition. See \cite{Skorohod.56}.
	\begin{lemma}\label{lem:J1IMpliesM1}
		If $\psi_n\overset{J_1}{\longrightarrow}\psi_0$, then $\psi_n\overset{M_1}{\longrightarrow}\psi_0$.
	\end{lemma}
	The following is essential.
	\begin{lemma}\label{lem:polish}
		The Skorohod space $\D(\R_+,\R)$ is Polish with respect to $J_1$ and $M_1$ topologies. 
	\end{lemma}
	
	For the $J_1$ topology this is well-known (see e.g. \cite{Billingsley.99}), but for the $M_1$ topology we refer the reader section 12.8 of \cite{Whitt.02} where an explicit metric is constructed. While this metric is useful for proving a space is Polish (in fact, such a metric is essential), we will only need to recall various continuous functions and so we do not describe the metric.
	
	In the sequel, we write $\operatorname{Id}$ for the identity function $\operatorname{Id}(t) = t$ on $\R_+$.

	\subsubsection{Some Properties}
	
	The first result we recall is trivial.
	\begin{lemma}\label{lem:weaklawFromCLT}
		Suppose that $c_n\to \infty$, $\kappa>0$, and $\psi_n,\psi$ are c\`adl\`ag functions such that 
		$ c_n(\psi_n - \kappa \operatorname{Id})\overset{M_1}{\longrightarrow} \psi.$ Then $\psi_n\overset{M_1}{\longrightarrow} \kappa \operatorname{Id}.$
	\end{lemma}
	
	The following is also observed in \cite{Skorohod.56}. See also Section 12.4 and 12.5 of \cite{Whitt.02} generally and Theorems 12.4.1 and 12.5.1 and Lemma 12.5.1 in particular.
	\begin{lemma}\label{lem:contpoints}
		Suppose that $\psi_n\to \psi$ in the $M_1$ topology. If $\psi(t) = \psi(t-)$ then $\psi_n(t-)\to\psi(t)$ and $\psi_n(t)\to \psi(t)$. Moreover, if the limit $\psi$ is continuous then $\psi_n\to \psi$ locally uniformly.
	\end{lemma}

	\subsubsection{Continuity of some functions}
	
	Let
	\begin{equation*}
		\psi^{\uparrow}(t) = \sup_{s\le t} \psi(s).
	\end{equation*}

	\begin{lemma} \label{lem:composition}
		Suppose that $\varphi_n\to\varphi$ and $\psi_n\to\psi$ in the $M_1$ topology. Then the following hold:
		\begin{enumerate}
			\item If 
			\begin{equation*}
				\left\{t: \Big(\psi(t)-\psi(t-)\Big) \Big(\varphi(t)-\varphi(t-)\Big) < 0\right\} = \emptyset
			\end{equation*} then $\varphi_n+\psi_n\to \varphi+\psi$ in the $M_1$ topology.
			\item  If $a_n\to a$ are real numbers then $a_n\psi_n \to a\psi$ in the $M_1$ topology.
			\item 
			If $\psi_n, \psi$ are non-decreasing. Then $\varphi_n\circ \psi_n\to \varphi\circ\psi$ in the $M_1$ topology if either $\psi$ is strictly increasing and continuous or $\varphi$ is continuous. 
			
			\item If $\psi_n$ and $\psi$ are non-decreasing, then $\psi^{-1}_n\to \psi^{-1}$ in the $J_1$ topology if $\psi$ is unbounded and strictly increasing at $0$.
			\item  $\psi_n^\uparrow \longrightarrow \psi^\uparrow$.
		\end{enumerate}
	\end{lemma}
	\begin{proof} The scaling results are obvious from the definitions so we only mention where to find proofs for the other results. For addition see  \cite[Corollary 12.7.1, Theorem 12.7.3]{Whitt.02}, for composition see  \cite[Theorem 13.2.3]{Whitt.02} and \cite[Section 7.2.2]{Whitt.02a}, for inverses \cite[Corollary 13.6.3]{Whitt.02} and for supremumums see \cite[Theorem 13.4.1]{Whitt.02}.
	\end{proof}
	
	\begin{remark}
		In \cite[Chapter 12]{Whitt.02}, Whitt introduces the strong and weak $M_1$ topologies ($SM_1$ and $WM_1$ respectively). There is a difference between the $SM_1$ and $WM_1$ on $\D(\R_+,\R^k)$ for $k\ge2$; however, they both agree with the $M_1$ topology on $\D(\R_+,\R)$. See Section 12.3 of \cite{Whitt.02} for more details.
	\end{remark}
	
	The following follows from the Lemma \ref{lem:contpoints} and part (5) of Lemma \ref{lem:composition} above.
	\begin{corollary}\label{cor:inf}
		Suppose that $\psi_n\to \psi$ in the $M_1$ topology. If $\psi$ is continuous at both $a<b$ then
		\begin{equation*}
			\inf_{s\in[a,b]} \psi_n(s) \to \inf_{s\in [a,b]} \psi(s).
		\end{equation*}
	\end{corollary}
	\subsubsection{Fluctuations of inverse} We now turn to fluctuations of the various functions above. We start with the fluctuations of the inverse. 
	In the continuous case see Vervaat \cite{Vervaat.72} and references therein, and in the general result below can be found in the (unindexed) corollary after Lemma 7.6 in \cite{Whitt.80}.
	\begin{lemma}\label{lem:passageFluctuations}
		Let $(\psi_n;n\ge 1)$ be a sequence non-decreasing c\`adl\`ag functions and $\kappa>0$ a constant. Suppose that there are constants $c_n\to\infty$ such that $c_n(\psi_n - \kappa \operatorname{Id})\overset{M_1}\longrightarrow \psi$
		for some c\`adl\`ag function $\psi$. Then 
		\begin{equation*}\label{eqn:appendixSecond}
			c_n(\psi_n^{-1} - \kappa^{-1}\operatorname{Id})\overset{M_1}\longrightarrow \left(-\kappa^{-1} \psi(t/\kappa); t\ge 0 \right).
		\end{equation*}
	\end{lemma}
	
	\subsubsection{Fluctuations of supremum}
	We will need the following result about supremums as well from Whitt \cite{Whitt.80}.

	\begin{lemma}[Whitt {\cite[Theorems 6.2, 6.3]{Whitt.80}}]\label{lem:sup}
		Let $(\psi_n;n\ge 1)$ be a sequence c\`adl\`ag functions and $\kappa>0$ a constant. Suppose that there are constants $c_n\to\infty$ such that $c_n(\psi_n - \kappa \operatorname{Id})\overset{M_1}\longrightarrow \psi$ for some c\`adl\`ag function $\psi$. Then it holds that
		\begin{equation*}
			c_n( \psi_n^{\uparrow} - \kappa \operatorname{Id})\operatorname{M_1}\longrightarrow \psi.
		\end{equation*}
		
	\end{lemma}

	\subsubsection{Fluctuations of Compositions}
	The following is about convergence of compositions.
	\begin{theorem}[Whitt {\cite[Theorem 13.3.3]{Whitt.02}}]\label{thm:compositionFluc}
		Suppose that $\psi_n$ are a sequence of c\`adl\`ag functions, $\psi$ is a strictly increasing with continuous derivative $\psi'$. Suppose that $\varphi_n$ are non-decreasing c\`adl\`ag functions and that $\varphi$ is a strictly increasing continuous function. Furthermore, suppose that there exists c\`adl\`ag functions $u, v$ and a sequence $c_n\to\infty$ such that $c_n (\psi_n-\psi) \overset{M_1}{\longrightarrow }   u$ {and}$c_n(\varphi_n-\varphi) \overset{M_1}{\longrightarrow} v.$
		If
		\begin{equation}\label{eqn:JumpSigns}
			\left\{t\ge 0: \Big( u\circ\varphi(t) - u\circ\varphi(t-)\Big)\cdot \Big(v(t)-v(t-)\Big)< 0\right\} = \emptyset,
		\end{equation}
		then
		\begin{equation*}
			c_n\left(\psi_n\circ \varphi_n - \psi\circ\varphi\right) \overset{M_1}{\longrightarrow} u\circ \varphi+ (\psi'\circ \varphi)\cdot v.
		\end{equation*}
	\end{theorem}
	
	In words, \eqref{eqn:JumpSigns} says that if $u\circ\varphi$ and $v$ jump at the same time $t$, then they only jump in the same direction. We will often only use the above theorem when $\psi$ and $\varphi$ are linear functions, and so we state the following corollary for ease of use later.
	
	\begin{corollary}\label{cor:threefunc}
		Suppose that $\psi_n,\varphi_n,\gamma_n$ are c\`adl\`ag functions, $\alpha,\beta,\delta>0$. Suppose that $c_n\to\infty$, $u,v,w$ are c\`adl\`ag functions with no negative jumps. If    \begin{equation*}
			c_n\left( \psi_n- \alpha\operatorname{Id}\right) \overset{M_1}{\longrightarrow} u,\qquad   c_n\left( \varphi_n-\beta \operatorname{Id}\right) \overset{M_1}{\longrightarrow} v,\qquad \textup{and}\qquad  c_n\left( \gamma_n-\delta \operatorname{Id}\right) \overset{M_1}{\longrightarrow} w,
		\end{equation*}
		then
		\begin{align*}
			c_n &\left( \psi_n\circ\varphi_n\circ\gamma_n(t) -  \alpha\beta\delta t;t\ge 0\right)\overset{M_1}{\longrightarrow} \left(u(\beta\delta t) + \alpha v(\delta t) + \alpha \beta w(t);t\ge 0\right).
		\end{align*}    
	\end{corollary}
	
	\subsection{Convergence of excursions}

	We now introduce a definition of good functions. This is essentially Definition 1 and 2 in \cite{DvdHvLS.20} combined. Recall the definition of $\EE_\infty$ from \eqref{eqn:EE_inftyDef} and let $\cY_\infty(\psi) = \{r: (l,r)\in \EE_\infty(\psi)\}$.
	\begin{definition}\label{def:good}
		We say that a c\`adl\`ag function $\psi$ is \textit{good} if
		\begin{enumerate}[i)]
			\item For all $(l,r)\in \EE_\infty(\psi)$, $\psi$ is continuous at $r$;
			\item The function $\psi$ does not attain a local minimum at any point $r\in\cY_\infty(\psi)$;
			\item The set $\cY_\infty(\psi)$ does not contain any isolated points;
			\item It holds that $\Leb\left(\R_+\setminus \bigcup_{e\in \EE_\infty(\psi)} e \right)  = 0$;
			\item For all $\eps>0$, there are finitely many $(l,r)\in \EE_\infty(\psi)$ such that $r-l\ge \eps$.
		\end{enumerate}
	\end{definition}
	\begin{remark}
		We include the assumption that $\psi$ is continuous at $r\in \cY_\infty(\psi)$ for simplifying subsequent proofs; however, the continuity follows from (ii)--(iv). See \cite[Remark 12]{DvdHvLS.20}.
	\end{remark}

	The following is an extension of Lemma 14 in \cite{DvdHvLS.20}, except we weaken the assumption to only $M_1$ convergence instead of $J_1$. It is useful to introduce analogues of $\EE_\infty, \cY_\infty,\LL_\infty^\downarrow$ on compact time horizons. Given a good function $\psi$, and a $T>0$ such that $\psi(T)>\inf_{s\le T} \psi(s)$ let
	\begin{align*}
		&\EE_T(\psi) = \{\{(l,r)\in \EE_\infty(\psi): r\le T\}\}, &
		&\LL_T^\downarrow (\psi) = \ORD(\{\{r-l: (l,r)\in \EE_T(\psi)\}\}),\\
		&\cY_T(\psi) = [0,T]\cap \cY_\infty(\psi).
	\end{align*}
	\begin{lemma}\label{lem:good}
		Let $(\psi_n;n\ge 1)$ and $\psi$ be a collection of c\`adl\`ag functions which do not jump downwards. Let $T\in(0,\infty)$ be such that $\psi(T) = \psi(T-)> \inf_{s\le T} \psi(s)$. Suppose that 
		\begin{enumerate}[i)]
			\item $\psi_n\overset{M_1}{\longrightarrow} \psi$;
			\item $\psi$ is good.
		\end{enumerate}
		Then $\LL^\downarrow_T(\psi_n){\longrightarrow } \LL^\downarrow_T(\psi)$ point-wise.
		
		If, in addition,
		\begin{enumerate}[i)]
			\setcounter{enumi}{2}
			\item  For all $\eps>0$, there exists an $N = N(\eps)<\infty$ and $T = T(\eps)<\infty$ such that 
			\begin{equation*}
				\sup\left\{r-l: (l,r)\in \EE_\infty(\psi_n), n\ge N, r\ge T \right\} <\eps;
			\end{equation*}
		\end{enumerate}
		then $\LL^\downarrow_\infty(\psi_n){\longrightarrow } \LL^\downarrow_\infty(\psi)$ point-wise.
	\end{lemma}

	\begin{proof} We prove the first statement. The second condition follows easily. Write
		\begin{equation*}
			\LL_T^\downarrow(\psi)  = (\zeta_1,\zeta_2,\dotsm)\qquad\textup{and}\qquad\forall n\ge 1\quad\LL_T^\downarrow(\psi_n)  = (\zeta_1^n,\zeta^n_2,\dotsm)
		\end{equation*} Fix some $k$ such that $\zeta_k>\zeta_{k+1}$, which is possible as $\sum_{j} \zeta_j < T$. It suffices to show for all $j\in[k]$ that the convergence $\zeta_j^n\to \zeta_j$ holds.
		In order to do this we will show that $\liminf_{n} \zeta_j^n \ge \zeta_j$ and $\limsup_n \zeta_j^n \le \zeta_j$. 
		
		We start with the $\liminf$. Fix some $\eps>0$ such that $\eps\ll \zeta_{k+1}$, which is possible since $\zeta_{k+1}>0$ as $\psi$ is good. Denote the corresponding excursion interval of length $\zeta_j$ by $(l_j,r_j)$. For each $j\in[k]$, we can find $\delta\in (0,\eps/2)$ such that $l_j + \delta, r_j-\delta$ are continuity points of $\psi$ for each $j\in[k]$. Such a $\delta$ exists as there are only countably many discontinuities of $\psi$. We can also find continuity points $t_j\in(l_j+\delta,r_j-\delta)$ for each $j$. Note that since $\psi$ is good, $l_j$ and $r_j$ are also continuity points of $\psi$. 
		
		Since $\psi$ is good we can find an $\eta>0$ such that for all $j\in[k]$
		\begin{equation}\label{eqn:onlyCareAboutRight1}
			\inf_{s\le l_j}\psi(s) +\eta = \inf_{s\le r_j}\psi(s) +\eta < \inf_{s\in[l_j+\delta, r_j-\delta]} \psi(s).
		\end{equation} As the interior of the excursion intervals of $\psi$ have full Lebesgue measure and since $t\mapsto \inf_{s\le t} \psi(t)$ is a non-increasing continuous function (recall that $\psi$ does not jump downwards) we can find an $\widetilde{r}_j\in[0,l_j]\cap \cY(\psi)$ such that 
		\begin{equation}\label{eqn:onlyCareAboutRight2}
			\inf_{s\le l_j} \psi(s) + \frac{\eta}{2} \ge \inf_{s\le \widetilde{r}_j} \psi(s) \ge \inf_{s\le l_j} \psi(s)
		\end{equation}

		By Lemma \ref{lem:contpoints} and Corollary \ref{cor:inf} we have
		\begin{align*}
			\begin{array}{l}\inf_{s\le u} \psi_n(s) \to \inf_{s\le u} \psi(s)\\
				\psi_n(u) \to \psi(u)\\
				\inf_{s\in[u,v]} \psi_n(s) \to \inf_{s\in[u,v]} \psi_n(s)
			\end{array}
			\qquad\forall u<v\in\{\widetilde{r}_j,l_j+\delta, t_j, r_j-\delta, r_j; j\in[k]\}
		\end{align*} as each such $u,v$ is a continuity point of $\psi$. In particular, using \eqref{eqn:onlyCareAboutRight1} and \eqref{eqn:onlyCareAboutRight2} we get
		\begin{align*}
			&\lim_{n\to\infty} \left(\inf_{s\in[l_j+\delta, r_j-\delta_j]} \psi_n(s) - \inf_{s\le \widetilde{r}_j} \psi_n(s) \right)\ge \frac{\eta}{2} \quad\textup{ and }\\
			&\lim_{n\to\infty} \left(\inf_{s\in[l_j+\delta, r_j-\delta_j]} \psi_n(s) - \inf_{s\le r_j} \psi_n(s) \right) \ge \frac{\eta}{2}.
		\end{align*}
		Hence, for all $n$ large enough, there exists an excursion interval of $\psi_n$ that straddles $t_j$ of length at least $\zeta_j-2\delta \ge \zeta_j-\eps$. Hence, 
		\begin{equation*}
			\liminf_{n} \zeta^n_j \ge \zeta_j,\qquad \forall j\in[k].
		\end{equation*}
		
		We now turn to the $\limsup$. To do this, fix an $\eps\ll \zeta_{k+1} - \zeta_{k}$. Now we can find some $r_1'<r_2'<\dotsm<r_M'\in \cY_T(\psi)$ such that (setting $r_0' = 0$) we have
		\begin{enumerate}
			\item There exists distinct $i_1,i_2,\dotms,i_k$ such that $r_{i_j}' - r_{i_j-1}'\le \zeta_j+\eps$;
			\item For each $p\notin\{i_1,\dotms,i_k\}$, $r_p' - r'
			_{p-1} \le \zeta_{k+1}+\eps< \zeta_k$;
			\item $r_M' \ge \sup\{\cY_T\} - \eps.$
		\end{enumerate}
		Indeed, the first follows from the fact that $\psi$ is good and hence each left excursion endpoint $l_j = \sup\{r\in \cY_T(\psi): r< l_j\}$ by property (iv) of the definition. The remaining two follow by similar logic.
		
		Now, by Corollary \ref{cor:inf} and the continuity of $\psi$ at each $r_j'$ we have
		\begin{equation*}
			\inf_{s\le u} \psi_n(s) \to \inf_{s\le r_j'}\psi(r_j) =  \psi(r_j')\qquad\textup{ for all }j\in[M].
		\end{equation*}
		As each $r_j'$ is not a local minimum of $\psi$, we see
		\begin{equation*}
			\psi(0)> \psi(r_1') >\dotsm > \psi(r_M')
		\end{equation*}
		and so for all $n$ sufficiently large we get
		\begin{equation}\label{eqn:boundsINfpsi}
			\inf_{s\le r_1'} \psi_n(s) > \inf_{s\le r_2'} \psi_n(s)> \dotsm > \inf_{s\le r_M'} \psi_n(s).
		\end{equation}
		We also see that for Lebesgue a.e. $\delta\in(0,T-r_M')$ that
		\begin{equation*}
			\inf_{s\in[r_{M}' + \delta, T]} \psi_n(s) > \psi(r_M') = \inf_{s\le r_M'} \psi(s).
		\end{equation*}
		We can clearly choose $\delta\in(\eps,2\eps)$ so that $r_M'+\delta$ is a continuity point of $\psi$ and conclude that for all $n$ large \eqref{eqn:boundsINfpsi} holds and
		\begin{equation*}
			\inf_{s\le r_M'} \psi_n(s) < \inf_{s\in[r_M'+\delta,T]} \psi_n(s).
		\end{equation*}
		This is enough to conclude that $\psi_n$ (for these large $n$) that $  \zeta_j^{n} \le \zeta_j+\eps$ for each $j\in[k]$. This establishes the claim.
	\end{proof}

	We will need a point process version of the preceding lemma. To handle this, given a point-set $\Xi\subset\R_+\times (\R^{k}_+\setminus\{0\})$ we say it is locally finite if for any compact $K \subset \R_+\times (\R_+^k\setminus \{0\})$ it holds that $\#(\Xi\cap K) <\infty.$ We will write $\Xi^{(n)}\to \Xi$ if the associated counting measures converge in the vague topology. The next lemma follows easily from Lemmas  \ref{lem:contpoints}, \ref{lem:good} above. See also \cite{Clancy.24}.
	\begin{lemma} \label{lem:ppconv} Suppose that $\psi_n\to\psi$ and $\varphi_n\to\varphi$ in the $M_1$ topology, where $\psi$ is good and $\varphi$ is strictly increasing.
		Suppose that $\Xi^{(n)} = \Xi^{(n),\varphi_n}$ and $\Xi = \Xi^{\varphi}$ are defined as
		\begin{align*}
			&\Xi^{(n)}:= \{(l,r-l, \varphi_n(r)-\varphi_n(l-)): (l,r)\in \EE_\infty(\psi_n)\},\\
			&\Xi = \{ (l,r-l,\varphi(r)-\varphi(l)): (l,r) \in \EE_\infty(\psi)\}.
		\end{align*} If $\varphi$ is continuous at each $t\in \{l,r: (l,r)\in\EE_\infty(\psi)\}$, then $\Xi^{(n)} \longrightarrow \Xi$.
	\end{lemma}
	
	We note that if $\Xi^{(n)}\to \Xi$ and $(t,x)\in \Xi$ then there exists $(t_n,x_n)\in \Xi^{(n)}$ (for all $n$ large) such that $(t_n,x_n)\to (t,x)$. See, for example, Proposition 4.26 in \cite{Maggi.12}. The next two results follow easily
	\begin{corollary}\label{cor:techcor}
		Suppose that $\psi_n\to\psi$ and $\varphi_n\to\varphi$ in the $M_1$ topology, where $\psi$ is good and $\varphi(t) = \gamma t$ for some $\gamma>0$. Suppose that $a_n\to a$ and $b_n\to b$ are such that $a+\gamma b >0$. 
		Suppose that ${\Xi}^{(n)}$ and ${\Xi}$ are defined as
		\begin{align*}
			&\Xi^{(n)}:= \left\{\Big(l, a_n\big(r-l\big)+ b_n\big( \varphi_n(r)-\varphi_n(l-)\big)\Big): (l,r)\in \EE_\infty(\psi_n)\right\},\\
			&\Xi := \left\{\Big (l,(a+\gamma b) (r-l) \Big): (l,r) \in \EE_\infty(\psi)\right\}.
		\end{align*} Then $\Xi^{(n)}\to \Xi$.
	\end{corollary}
	
	\begin{corollary}\label{cor:techcor2}
		Suppose that $(\psi_n,\varphi_n)\to (\psi,\varphi)$ in the $M_1$ topology, where $\psi$ is good, $\varphi$ is strictly increasing and continuous at each $t\in \{l,r:(l,r)\in\EE_\infty(\psi)\}$. Suppose that $a_n\to a$ and $b_n\to 0$ are real numbers.
		Suppose that ${\Xi}^{(n)}$ and ${\Xi}$ are defined as
		\begin{align*}
			&\Xi^{(n)}:= \left\{\Big(l, a_n\big(r-l\big)+ b_n\big( \varphi_n(r)-\varphi_n(l-)\big)\Big): (l,r)\in \EE_\infty(\psi_n)\right\},\\
			&\Xi := \left\{\Big (l,a(r-l) \Big): (l,r) \in \EE_\infty(\psi)\right\}.
		\end{align*} Then $\Xi^{(n)}\to\Xi$.
	\end{corollary}

	\section{Weak convergence of $V$} \label{sec:WeakConv}
	
	In this section we establish three different scaling limits for the exploration processes. These are contained in Propositions \ref{prop:vconvClass}, \ref{prop:vconvINter}, and \ref{prop:vconvbip}. In order to prove the results we recall a recent result by Limic \cite{Limic.19} which implies weak convergence of $X_n^{j,i}$ and $\widetilde{X}_n^{j,i}$ modulo recentering and rescaling. This is key to essentially all the subsequent analysis. For an application of part (4) of Lemma \ref{lem:composition}, we need to recall some elementary path properties of thinned L\'{e}vy processes.
	
	\subsection{Some Scaling Results}

	Recall that given a sequence $\bz = \bz^{(n)}\in\ell^f_\downarrow$ of finite length vectors we write $\bz\rightsquigarrow (\beta,\btheta)$ if \eqref{eqn:sigma2tozer}, \eqref{eqn:sigma3order}, and \eqref{eqn:sigmahumbs} hold. 
	Write $N^\bz =N^{\bz,q} = (N^\bz(t);t\ge 0)$ by
	\begin{equation*}
		N^\bz(t) = \sum_{j=1}^{\len(\bz)} z_j 1_{[\xi_j\le qt]},
	\end{equation*}
	where $\xi_j\sim \Exp(z_j)$ are independent exponential random variables. Also recall that $W^{\beta,\btheta,\lambda}$ is defined in \eqref{eqn:thinnedLevy}.
	
	\begin{theorem}[Limic \cite{Limic.19}]\label{thm:limic}
		Suppose that $\bz\rightsquigarrow (\beta,\btheta)$ and $q = \frac{1}{\sigma_2(\bz)}+ \lambda$. Let $c_n \sigma_2(\bz)\to 1$. Then in the $J_1$ (and hence $M_1$) topology
		\begin{equation*}
			c_n (N^\bz - \operatorname{Id})\weakarrow W^{\beta,\btheta,\lambda}.
		\end{equation*}
	\end{theorem}
	
	\begin{remark}
		This is not how the result is stated in \cite{Limic.19}. Therein, Limic assumes that $(\beta,\btheta)\in \cI$ instead of $\cI^\circ$ as she is primarily concerned with the eternal multiplicative coalescence, which requires that $(\beta,\btheta)\in \cI$. However, the proof that Limic gives in \cite{Limic.19} (and the supplement \cite{Limic.19_Supplement}) implies the result for all $(\beta,\btheta)\in \cI^\circ$ as well.
	\end{remark}
	
	The following corollary follows easily from Theorem \ref{thm:limic} and Lemma \ref{lem:composition}.
	\begin{corollary}
		\label{cor:Limic}
		Suppose that $\bz\rightsquigarrow (\beta,\btheta)$ and $q = \frac{\kappa}{\sigma_2(\bz)}  + \lambda_n$ where $c_n\sigma_2(\bz)\to 1$ and $\lambda_n\to \lambda$. Then in the $M_1$ topology
		\begin{equation*}
			c_n \left(N^\bz  - \kappa \operatorname{Id}\right) \weakarrow \left( W^{\beta,\btheta,0}(\kappa t) +  \lambda t ;t\ge 0 \right).
		\end{equation*}
	\end{corollary}

	Examining \eqref{eqn:Xdef}, we will want to understand the asymptotics of $q_{12}/q_{jj}$ for each $j$. Ignoring the $o(1)$ terms in equation \eqref{eqn:Qextension} and $o(c_n^{-1})$ term in \eqref{eqn:weightWindow} which are both negligible in the limit, in the limit we have
	\begin{align*}
		\frac{q_{12}}{q_{22}} &= \frac{\kappa_{12} c_n + \lambda_{12}}{\kappa_{22} \sigma_2(\bw^1)^{-1} + \lambda_{22}} = \frac{\kappa_{12} +\lambda_{12}c_n^{-1}}{\kappa_{22} \frac{1}{\sigma_{2}(\bw^2)c_n} + \lambda_{11} c_n^{-1}}\\
		&=\frac{\kappa_{22} +\lambda_{12}c_n^{-1}}{\kappa_{22} (1+\alpha c_n^{-1}) + \lambda_{22} c_n^{-1}} = \frac{\kappa_{12}}{\kappa_{22}} + \frac{1}{\kappa_{22}}\left(\lambda_{12}  - \frac{\alpha \kappa_{12} + \lambda_{22}\kappa_{12}}{\kappa_{22}}\right)c_n^{-1} + o(c_n^{-1}). 
	\end{align*} A similar argument holds for $q_{12}/q_{11}$. Using the fact that $GL_2(\R)\subset \R^{2\times 2}$ is an open subset and $A\mapsto A^{-1}$ is continuous on $GL_2(\R)$ we get the following lemma:
	\begin{lemma}\label{lem:q12/qiiApprox}
		Suppose Assumption \ref{ass:main}. 
		In either Regime \ref{regime:classic} or Regime \ref{regime:interacting} then
		\begin{align*}
			\frac{q_{12}}{q_{11}} &= \frac{\kappa_{12}}{\kappa_{11}} + \frac{\alpha \kappa_{12}\kappa_{11}+ \lambda_{12}\kappa_{11} - \lambda_{11}\kappa_{12}}{\kappa_{11}^2} c_n^{-1} + o(c_n^{-1})\\
			\frac{q_{12}}{q_{22}} &= \frac{\kappa_{12}}{\kappa_{22}} + \frac{-\alpha \kappa_{12}\kappa_{22}+ \lambda_{12}\kappa_{22} - \lambda_{22}\kappa_{12}}{\kappa_{22}^2} c_n^{-1} + o(c_n^{-1}).
		\end{align*}
		In particular,  $      R = \operatorname{diag}(Q)^{-1} Q \longrightarrow \operatorname{diag}(K)^{-1}K.
		$
	\end{lemma}
	
	Recall $J^\bccc$ defined in \eqref{eqn:Jcreg}. In \cite{AL.98}, it is shown that the process is well-defined and basic scaling properties of their excursion lengths are noted. We include such a relation here.
	\begin{lemma}\label{lem:jumplevypart}
		For any $\bccc\in \ell^3_\downarrow$ and $a>0$
		\begin{equation*}
			\left(J^\bccc(at);t\ge 0\right) \overset{d}{=} (a^{-1} J^{a\bccc}(t);t\ge 0).
		\end{equation*}
		More generally, for any $(\beta,\btheta, \lambda)\in \cI^\circ\times \R$ and any $a>0$
		\begin{equation*}
			\left(a W^{\beta,\btheta,\lambda}(at) ;t\ge 0\right)\overset{d}{=} \left(W^{a^3\beta,a\btheta,a^2\lambda}(t);t\ge 0\right).
		\end{equation*}
	\end{lemma}
	\begin{proof}
		The first follows from standard properties of exponential random variables. For the second note that for a standard Brownian motion $B$
		\begin{align*}
			&\left(W^{\beta,\btheta,\lambda}(at);t\ge 0\right) \overset{d}{=} \left(\sqrt{a\beta} B(t) - \frac{a^2\beta}{2}t^2+ \lambda a t + \frac{1}{a} J^{a\btheta}(t);t\ge 0\right)\\
			&= \left(\frac{1}{a}  \left(\sqrt{a^3\beta} B(t) - \frac{a^3\beta}{2}t^2 + \lambda a^2 t +J^{a\btheta}(t)\right);t\ge 0\right)=a^{-1} W^{a^3\beta,a\btheta, a^2 \lambda}.
		\end{align*}
	\end{proof}

	The above lemma has the following useful corollary.
	\begin{corollary}\label{cor:vectorvalue}
		Suppose that $(\beta_j,\btheta_j)\in \cI^\circ$ with $\sum_{j} (\beta_j,\btheta_j)\in \cI$. Fix $\Lambda$, a symmetric $2\times 2$ matrix. For any $\bv\in (0,\infty)^2$ let 
		\begin{equation}
			\label{eqn:betabv} \beta^\bv = v_1^3\beta_1+v_2^3 \beta_2,\qquad \btheta^\bv = v_1\btheta_1 \bowtie v_2\btheta_2,\qquad \lambda^\bv = \langle \bv, \Lambda\bv\rangle.
		\end{equation} If $\bu,\bv\in (0,\infty)^2$ are linearly dependent then
		\begin{equation*}
			\bzeta^{\beta^\bv, \btheta^\bv, \lambda^\bv} \bv \overset{d}{=} \bzeta^{\beta^\bu,\btheta^\bu,\lambda^\bu}\bu.
		\end{equation*} 
	\end{corollary}

	The next lemma is useful in the sequel and follows from standard properties of Brownian motions. We omit the proof.
	\begin{lemma}\label{lem:summingtwothinned}
		Suppose that $W^{\beta_j,\btheta_j,\lambda_j}_j$, $j\in[2]$, are independent thinned L\'{e}vy processes defined in \eqref{eqn:thinnedLevy}. Then
		\begin{equation*}
			W_1^{\beta_1,\btheta_1,\lambda_1}+W_2^{\beta_2,\btheta_2,\lambda_2}\overset{d}{=} W^{\beta,\btheta,\lambda}
		\end{equation*}
		where $\beta= \beta_1+\beta_2, \lambda = \lambda_1+\lambda_2$ and $\btheta = \btheta_1\bowtie\btheta_2$.
	\end{lemma}
	
	\subsection{Some Fluctuation Theory}
	In this section we recall some basic properties of thinned L\'{e}vy processes $W^{\beta,\btheta,\lambda}$ from \eqref{eqn:thinnedLevy}.

	\begin{lemma}[Aldous and Limic {\cite[Proposition 14]{AL.98}}]\label{lem:Zpath}
		Suppose $(\beta,\btheta)\in \cI$ and let $Z(t) = W^{\beta,\btheta,\lambda}(t)$. Then
		\begin{enumerate}[(a)]
			\item $Z(t)\to -\infty$ in probability as $t\to\infty$;
			\item $\PR(Z(t)> \inf_{s\le t} Z(s) ) = 1$ for all $t> 0$.
			\item $\sup\{r-l: (l,r)\in \EE_\infty(Z); l\ge t\} \to 0$ in probability as $t\to\infty$;
			\item With probability $1$, $\{r: (l,r)\in \EE_\infty(Z)\}$ has no isolated points;
			\item If $(l_1,r_1), (l_2,r_2)\in \EE_\infty(Z)$ and $l_1<l_2$ then $Z(r_2)< Z(r_1).$
		\end{enumerate}
	\end{lemma}
	
	The following lemma was observed in \cite[Lemma 15]{DvdHvLS.20} in the case where $\beta = 0$. The same proof works in general. 
	
	\begin{lemma}[Dhara et al.\cite{DvdHvLS.20}]\label{lem:Zgood}
		With probability $1$, $W^{\beta,\btheta,\lambda}$ is good on $\R_+$ for any $(\beta,\btheta,\lambda)\in \cI\times \R$.    In particular, if $(\beta,\btheta,\lambda)\in \cI\times \R$, then with probability $1$, $W^{\beta,\btheta,\lambda}$ is continuous at $r$ for each $r\in\cY_\infty(W^{\beta,\btheta,\lambda})$.
	\end{lemma}
	
	Finally, we will need one other property. 
	\begin{lemma}\label{lem:decreaseat0}
		Let $Z = W^{\beta,\btheta,\lambda}$ for some $(\beta,\btheta,\lambda)\in \cI\times \R$. Then a.s.
		\begin{equation*}
			\inf\{t: Z(t)< 0\} = 0.
		\end{equation*}
	\end{lemma}
	
	\begin{proof}
		This is established in \cite{AL.98} while proving a different result. Let $I_p(t)$ be rate $\theta_p$ independent rate Poisson processes and let $B$ be a standard Brownian motion. Let 
		\begin{equation*}
			L(t) = \sqrt{\beta}B(t) + \lambda t + \sum_{p} \theta_p(I_p(t)-\theta_pt).
		\end{equation*}One can easily couple $Z$ and $L$ so that $Z(t)\le L(t)$ for all $t\ge 0$. By Theorem VII.1 in \cite{Bertoin.96}, $ \PR(\inf\{t: L(t)<0\} = 0) = 1.$ The result now follows.
	\end{proof}

	\subsection{Convergence of $V$ in regime \ref{regime:classic}}
	
	\begin{proposition}\label{prop:vconvClass}
		Suppose Assumption \ref{ass:main} and Regime \ref{regime:classic}. Then jointly in the $M_1$ topology
		\begin{align}
			\nonumber &\left(c_n V_n(t);t\ge 0\right) \weakarrow \frac{1}{\kappa_{11}^2} W^{\beta^\bv,\btheta^\bv,\lambda^\bv}\\
			\label{eqn:Un2compXn21}&\left( U_n^2 \circ X_{n}^{2,1}(t);t\ge 0\right) \weakarrow \left(\frac{\kappa_{11} \kappa_{12}}{\kappa_{22}(1-\kappa_{22})}t;t\ge 0\right).
		\end{align}
		where $\bv = \begin{bmatrix}
			\kappa_{11}\\ \frac{\kappa_{11}\kappa_{12}}{1-\kappa_{22}}
		\end{bmatrix}$ and $\beta^\bv,\btheta^\bv, \lambda^\bv$ are defined via \eqref{eqn:betabv}.

	\end{proposition}
	\begin{proof}
		By Corollary \ref{cor:Limic}, and Skorohod's representation theorem, we can assume that we are working on a probability space such that
		\begin{align*}
			&\left( c_n (N^{\bw^1}(t) - \kappa_{11}t);t\ge 0\right) \overset{M_1}{\longrightarrow} (W_1^{\beta_1,\btheta_1,0}(\kappa_{11}t) + \lambda_{11}t;t\ge 0)\\
			& \left( c_n (N^{\bw^2}(t) - \kappa_{22}t);t\ge 0\right) \overset{M_1}{\longrightarrow} (W_2^{\beta_2,\btheta_2,0}(\kappa_{22}t) + \lambda_{22}t;t\ge 0),
		\end{align*}
		where the two limits are independent of each other. Let us call the two limits above $Z_1$ and $Z_2$, respectively. 
		Recall \eqref{eqn:Xdef}. Therefore, by Lemma \ref{lem:q12/qiiApprox} and Lemma \ref{lem:composition}(2),
		\begin{align}
			\label{eqn:XiiconvClas}&\left(c_n (X_n^{i,i}(t) + (1-\kappa_{ii})t);t\ge 0\right)\overset{M_1}{\longrightarrow} Z_i\\
			\nonumber &\left(c_n(X_n^{j,i}(t) - \frac{q_{12}}{q_{jj}} \kappa_{ii} t;t\ge 0\right) \overset{M_1}{\longrightarrow} \left(\frac{\kappa_{12}}{\kappa_{jj}} Z_i(t);t\ge 0\right).
		\end{align}
		Using the approximations in Lemma \ref{lem:q12/qiiApprox} we have
		\begin{align}
			\label{eqn:x21class}  &\left(c_n(X_n^{2,1}(t) - \frac{\kappa_{12}\kappa_{11}}{\kappa_{22}} t;t\ge 0\right) \overset{M_1}{\longrightarrow} \left(\frac{\kappa_{12}}{\kappa_{22}} Z_1(t) + \mu_1t;t\ge 0\right),\\
			\nonumber  &\left(c_n(X_n^{1,2}(t) - \frac{\kappa_{12}\kappa_{22}}{\kappa_{11}} t;t\ge 0\right) \overset{M_1}{\longrightarrow} \left(\frac{\kappa_{12}}{\kappa_{11}} Z_2(t) + \mu_2t;t\ge 0\right),
		\end{align}
		where
		\begin{align*}
			\mu_{1} = \frac{\lambda_{12}-\alpha \kappa_{12}}{\kappa_{22}}\kappa_{11} - \frac{\lambda_{22}\kappa_{12}}{\kappa_{22}^2}\kappa_{11}&&\textup{and}&&&\mu_{2} = \frac{\lambda_{12}+\alpha \kappa_{12}}{\kappa_{11}}\kappa_{22} - \frac{\lambda_{11}\kappa_{12}}{\kappa_{11}^2}\kappa_{22}.
		\end{align*}
		By Lemma \ref{lem:sup} we see that
		\begin{align}
			\nonumber & \left( c_n (L_n^2(t) - (1-\kappa_{22})t );t\ge 0\right) \overset{M_1}{\longrightarrow}(-Z_2(t);t\ge 0).
		\end{align}
		Since $U^2_n$ is the first passage time of $L_n^2$ we see can apply Lemma \ref{lem:passageFluctuations} to get
		\begin{align}
			\label{eqn:uconv}  &\left(c_n( U_n^2 (t) - \frac{t}{1-\kappa_{22}}) 
			;t\ge 0\right)\overset{M_1}{\longrightarrow} \left(\frac{1}{1-\kappa_{22}} Z_2\left(\frac{t}{1-\kappa_{22}}\right);t\ge 0\right) .
		\end{align} Note that \eqref{eqn:x21class}, \eqref{eqn:uconv}, and Lemmas \ref{lem:weaklawFromCLT} and \ref{lem:composition} imply the desired convergence of $U_n^2\circ X^{2,1}_n(t)$ in \eqref{eqn:Un2compXn21}
		
		To apply Corollary \ref{cor:threefunc} in a more readable way, let us set 
		\begin{equation*}
			a_{12} = \frac{\kappa_{12}\kappa_{22}}{\kappa_{11}} \qquad\qquad a_{21} = \frac{\kappa_{12}\kappa_{11}}{\kappa_{22}}\qquad \qquad a_{22} = \frac{1}{1-\kappa_{22}}.
		\end{equation*}
		Note that by the eigenvalue assumption in Regime \ref{regime:classic} it holds that $a_{12}a_{22}a_{21} = 1-\kappa_{11}$. By Corollary \ref{cor:threefunc} and the fact that $Z_i$ do not jump downwards
		\begin{align*}
			&\left(c_n \left(X^{1,2}_n\circ U_n^2 \circ  X^{2,1}_n(t) - (1-\kappa_{11})t\right);t\ge 0\right)\\
			&\overset{M_1}\longrightarrow \bigg(\frac{\kappa_{12}}{\kappa_{11} }Z_2(a_{22}a_{21} t) + \mu_{2} a_{22}a_{21} t\\
			&\qquad\qquad+  a_{12}a_{22} Z_2(a_{22}a_{21}t)\\
			&\qquad \qquad\qquad +a_{12} a_{22} \frac{\kappa_{12}}{\kappa_{22}}Z_1(t) + a_{21}a_{22}\mu_1 t;t\ge 0\bigg).
		\end{align*}
		By using Lemma \ref{lem:composition}(1) we have
		\begin{align*}
			c_n V_n \overset{M_1}{\longrightarrow} \bigg(\left(\frac{\kappa_{12}}{\kappa_{11}}+a_{12}a_{22}\right)& Z_{2}(a_{22}a_{21} t) + \left(\frac{a_{12}a_{22}\kappa_{12}}{\kappa_{22}}+1 \right)Z_1(t)\\
			&+(a_{12}a_{22}\mu_{1} + \mu_{2}a_{22}a_{21}) t;t\ge0 \bigg).
		\end{align*} Call the limit above $V_\infty(t)$.
		To simplify this note
		\begin{align*}
			&\frac{\kappa_{12}}{\kappa_{11}} + a_{12}a_{22} = \frac{\kappa_{12}}{\kappa_{11}(1-\kappa_{22})}\\
			&\frac{a_{12}a_{22}\kappa_{12}}{\kappa_{22}} +1 = \frac{\kappa_{12}^2}{\kappa_{11}(1-\kappa_{22})}+1 = \frac{1-\kappa_{11}}{\kappa_{11}}+1 = \frac{1}{\kappa_{11} }\\
			& a_{12}a_{22}\mu_{1}+ \mu_{2} a_{22} a_{21} \\
			&\qquad = \frac{1}{1-\kappa_{22}} \left((\lambda_{12}-\alpha \kappa_{12})\kappa_{12} - \frac{\lambda_{22}\kappa_{12}^2}{\kappa_{22}} + (\lambda_{12}+\alpha \kappa_{12})\kappa_{12} - \frac{\lambda_{11}\kappa_{12}^2}{\kappa_{11}}\right)\\
			&\qquad =\frac{1}{1-\kappa_{22}}\left(2\lambda_{12}\kappa_{12} - \left(\frac{\lambda_{22}}{\kappa_{22}}+\frac{\lambda_{11}}{\kappa_{11}} \right)\kappa_{12}^2\right).
		\end{align*}
		Therefore,
		\begin{align*}
			V_\infty(t) &=  \frac{\kappa_{12}}{\kappa_{11}(1-\kappa_{22})} \left(W_2^{\beta_2,\btheta_2,0}\left(\frac{\kappa_{12}\kappa_{11}}{1-\kappa_{22}} t\right)+ \lambda_{22} \frac{\kappa_{12}\kappa_{11}}{\kappa_{22} (1-\kappa_{22})} t\right) \\
			&\qquad + \frac{1}{\kappa_{11}}\left(W_1^{\beta_1,\btheta_2,0}(\kappa_{11}t) + \lambda_{11}t\right)+\frac{1}{1-\kappa_{22}}\left(2\lambda_{12}\kappa_{12} - \left(\frac{\lambda_{22}}{\kappa_{22}}+\frac{\lambda_{11}}{\kappa_{11}} \right)\kappa_{12}^2\right)t\\
			&= \frac{\kappa_{12}}{\kappa_{11}(1-\kappa_{22})} W_2^{\beta_2,\btheta_2,0}\left(\frac{\kappa_{12}\kappa_{11}}{1-\kappa_{22}} t\right) + \frac{1}{\kappa_{11}} W_1^{\beta_1,\btheta_1,0}(\kappa_{11}t)\\
			& \qquad + \left(\lambda_{11} +\frac{2\lambda_{12}\kappa_{12}}{1-\kappa_{22}} + \frac{\lambda_{22} \kappa_{12}^2}{(1-\kappa_{22})^2} \right)t.
		\end{align*}
		Recall $\bv =\begin{bmatrix}\kappa_{11}\\  \frac{\kappa_{12}\kappa_{11}}{1-\kappa_{22}}.
		\end{bmatrix} = \begin{bmatrix}
			v_1\\v_2
		\end{bmatrix}$. Observe that we have shown
		\begin{align*}
			V_\infty(t) = \frac{1}{\kappa_{11}^2} v_2 W_2^{\beta_2,\btheta_2,0}(v_2 t) + \frac{1}{\kappa_{11}^2}  v_1 W_1^{\beta_1,\btheta_1,0}(v_1t) + \frac{1}{\kappa_{11}^2} \langle \bv ,\Lambda \bv\rangle t
		\end{align*} where we used for the drift term \begin{equation*}
			\frac{1}{\kappa_{11}^2} \langle \bv ,\Lambda \bv\rangle = \lambda_{11} + \frac{2\lambda_{12}\kappa_{12}}{1-\kappa_{22}} + \frac{\lambda_{22}\kappa_{12}^2}{(1-\kappa_{22})^2}.
		\end{equation*}By using Lemma \ref{lem:jumplevypart} we see
		\begin{align*}
			\left( V_\infty(t);t\ge 0\right) \overset{d}{=} \left(\frac{1}{\kappa_{11}^2} {W}_2^{\widetilde{\beta}_2, \widetilde{\btheta}_2,0}(t) + \frac{1}{\kappa_{11}^2}W_1^{\widetilde{\beta}_1,\widetilde{\btheta}_1,0} (t) + \frac{1}{\kappa_{11}^2}\langle \bv, \Lambda\bv\rangle t\right),
		\end{align*}
		where $\langle \bx,\by\rangle$ is the standard inner product on $\R^2$, $
		\widetilde{\beta}_j = v_j^3 \beta_j$ and $\widetilde{\btheta}_j = v_j\btheta_j.$
		The result now follows from Lemma \ref{lem:summingtwothinned}.
	\end{proof}
	
	The following lemma is easy to verify. 
	\begin{lemma}
		The vector $\bv$ from Proposition \ref{prop:vconvClass} is a right PF eigenvector of $K$.     
	\end{lemma}
	
	Before moving to the interacting regime we prove the following lemma.
	\begin{lemma}\label{lem:jointtandu}
		Suppose that Assumption \ref{ass:main} and Regime \ref{regime:classic} hold. If $s_n\to s$, $t_n\to t$, and $R_n = \operatorname{diag}(Q)^{-1} Q$ then
		\begin{align*}
			\begin{bmatrix}
				t_n-s_n\\
				U_n^{2}\circ X_n^{2,1}(t_n)- U_n^2\circ X_{n}^{2,1}(s_n-)
			\end{bmatrix} \weakarrow  (t-s)\begin{bmatrix}
				1\\
				\frac{\kappa_{11}\kappa_{12}}{\kappa_{22}(1-\kappa_{22})}
			\end{bmatrix}.
		\end{align*}
		\begin{align*}
			R_n^{-1} \begin{bmatrix}
				t_n-s_n\\
				U_n^{2}\circ X_n^{2,1}(t_n)- U_n^2\circ X_{n}^{2,1}(s_n-)
			\end{bmatrix} \weakarrow  (t-s)\bv.
		\end{align*} where $\bv$ is as in Proposition \ref{prop:vconvClass}.
	\end{lemma}
	\begin{proof} 
		Using Lemma \ref{lem:q12/qiiApprox} and the continuity of $A\mapsto A^{-1}$ we have $
		R_n^{-1} \to K^{-1} \operatorname{diag}(K)$.
		By Lemma \ref{lem:contpoints} and Proposition \ref{prop:vconvClass} 
		\begin{equation*}
			\begin{bmatrix}
				t_n-s_n\\
				U_n^{2}\circ X_n^{2,1}(t_n)- U_n^2\circ X_{n}^{2,1}(s_n-)
			\end{bmatrix} \weakarrow (t-s)\begin{bmatrix}
				1\\
				\frac{\kappa_{11}\kappa_{12}}{\kappa_{22} (1-\kappa_{22})}
			\end{bmatrix}.
		\end{equation*}
		Hence
		\begin{align*}
			R_n^{-1} \begin{bmatrix}
				t_n-s_n\\
				U_n^{2}\circ X_n^{2,1}(t_n)- U_n^2\circ X_{n}^{2,1}(s_n-)
			\end{bmatrix} \weakarrow (t-s) K^{-1} \begin{bmatrix}
				\kappa_{11}\\
				\frac{\kappa_{11}\kappa_{12}}{1-\kappa_{22}}
			\end{bmatrix} = (t-s)\bv.
		\end{align*}
		Indeed, $\bv$ is an eigenvector of $K$ with eigenvalue $1$ and $K$ is invertible so $\bv$ is an eigenvector of eigenvalue $1$ of $K^{-1}$ as well.
	\end{proof}

	We now turn to the following lemma.
	\begin{lemma}\label{lem:vecforClassic} Suppose $\bW =\bW^{(n)}$ and $Q = Q^{(n)}$ are as in Regime \ref{regime:classic}. Let $(\bcM_l^{(n)};l\ge 1)$ be the vector-valued component masses of $\G(\bW,Q)$ ordered so that $\bcM_{l,1}^{(n)}\ge \bcM_{l,2}^{(n)}\ge\dotsm$. 
		Then
		\begin{equation*}
			\left(\bcM_{p,j}^{(n)};l\ge 1\right) \weakarrow \left(\zeta_p^\creg u_j^\creg;p\ge 1\right)\qquad\begin{array}{l}
				\textup{ in }\ell^2_\downarrow\textup{ for }j =1\textup{ and }\\
				\textup{pointwise }\textup{for }j =2\\
			\end{array}
		\end{equation*}
		where $\bzeta^\creg$ is as in Theorem \ref{thm:classic}.
	\end{lemma}
	\begin{proof}
		Let $\cC_p^{(n)}$ be the connected component corresponding to the weight vector $\bcM_p^{(n)}$.  Each component of strictly positive mass $\bcM_{p,1}^{(n)}>0$ has a corresponding excursion interval of $V_n$ which we denote by $(l_{(p),n},r_{(p),n})$.
		
		For all $n$ large $R_n$ is invertible, and so Theorem \ref{thm:ckl1} implies
		\begin{equation}\label{eqn:asypmprop}
			\bcM_{p}^{(n)} = R_n^{-1} \begin{bmatrix}
				r_{(p),n}-l_{(p),n}\\
				U_n^2\circ X_n^{2,1} (r_{(p),n})- U_n^{2}\circ X_n^{2,1}(l_{(p),n}-)
			\end{bmatrix}.
		\end{equation}
		
		Note that by Proposition \ref{prop:vexcursions}, the excursions of $V_n$ in Regime \ref{regime:classic} appear in a size-biased order according to the size $\bcM_{p,1}^{(n)}$ (the scaling by $q_{11}$ does not affect the size-biasing). Aldous' general theory of size-biased point processes \cite{Aldous.97} extends to our setting by the technical Lemma \ref{lem:ppconv} and Corollary \ref{cor:techcor}. Therefore, Proposition \ref{prop:vconvClass} and Lemma \ref{lem:jointtandu} imply
		\begin{equation}\label{eqn:ell2bsmp1forClassic}
			\left(\bcM_{p;1}^{(n)};p\ge 1\right) \weakarrow \ORD\left((r-l) v_1 ; (l,r)\in \EE_\infty(Z^{\beta^\bv, \btheta^\bv, \lambda^{\bv}} )\right)\qquad \textup{ in }\ell^2_\downarrow.
		\end{equation} 
		By Lemma \ref{lem:jointtandu}, and \eqref{eqn:asypmprop} it follows that for each $p\ge 1$
		\begin{equation*}
			\lim_{n\to\infty} \cM_{p,j}^{(n)} = v_j \lim_{n\to\infty} (r_{(p),n}-l_{(p),n}) .
		\end{equation*}
		Therefore, the $\ell^2_\downarrow$ convergence in \eqref{eqn:ell2bsmp1forClassic} implies
		\begin{equation*}
			\left(\bcM_{p;2}^{(n)};p\ge 1\right) \weakarrow \ORD\left((r-l) v_2 ; (l,r)\in \EE_\infty(Z^{\beta^\bv, \btheta^\bv, \lambda^{\bv}} )\right)\qquad \textup{ pointwise.}
		\end{equation*}
		The stated result now follows by an application of the scaling Corollary \ref{cor:vectorvalue}.
	\end{proof}

	\subsection{Convergence of $V$ in regime \ref{regime:interacting}}
	
	\begin{proposition}\label{prop:vconvINter}
		Suppose Assumption \ref{ass:main} and $\bW$ and $Q$ are as in Regime \ref{regime:interacting}. Then in the jointly in the $M_1$ topology
		\begin{align*}
			&\left( c_n V_n(t);t\ge 0\right) \weakarrow \left( Z^\ireg (t);t\ge 0\right)\\
			&\left(U_n^{2}\circ X_n^{2,1}(t);t\ge 0\right) \weakarrow (\inf\{u: Z_2(s)<-\lambda_{12}t\};t\ge 0)
		\end{align*}
		were $Z^\ireg$, $Z_2$ are defined in \eqref{eqn:ZiIregDef} and \eqref{eqn:Zidef}, respectively.
	\end{proposition}
	\begin{proof}
		The proof is analogous to the proof of Proposition \ref{prop:vconvClass}. By the same reasoning found therein, we can assume without loss of generality that for $i\in[2]$ and $j\neq i$
		\begin{equation}\label{eqn:xiiInteracting}
			\left(c_nX_n^{i,i}(t);t\ge 0\right) \overset{M_1}{\longrightarrow} \left(Z_i(t) ;t\ge 0\right)\qquad \left(c_n X_n^{j,i}(t);t\ge 0\right)\overset{M_1}{\longrightarrow} \left(\lambda_{12}t;t\ge 0\right)
		\end{equation} where $Z_i(t)$ is defined as in \eqref{eqn:Zidef}. Indeed, this is precisely \eqref{eqn:XiiconvClas} and \eqref{eqn:x21class} with $\kappa_{ii} = 1$ and $\kappa_{12} = 0$.
		Let us set $\overline{X}_n^{j,i}(t) = c_n X_n^{j,i}(t)$.
		
		Let $\overline{L}_n^{j}(t) = \inf_{s\le t} \overline{X}_n^{j,j}(s) = \sup_{s\le t} (-\overline{X}_n^{j,j}(s))$. By Lemma \ref{lem:composition}(5) we get
		\begin{equation*}
			\overline{L}_n^j\overset{M_1}{\longrightarrow} \left(-\inf_{s\le t}Z_j(s);t\ge 0\right).
		\end{equation*}
		Let $\overline{U}_n^j(t) = \inf\{s:\overline{L}_n^{j}(s)>t\}$. Observe
		$U_n^2(t/c_n) = \overline{U}^{2}_n(t)$
		and hence
		\begin{equation}\label{eqn:tildexij}
			c_n X_n^{1,2}\circ U_n^{2}\circ X_n^{2,1} = \overline{X}_n^{1,2}\circ \overline{U}_n^2 \circ \overline{X}_n^{2,1}.
		\end{equation} By Lemma \ref{lem:decreaseat0} and Lemma \ref{lem:composition}(4), we get
		\begin{equation*}
			\left(\overline{U}_n^2(t);t\ge 0\right) \overset{M_1}\longrightarrow \left(\inf\{s: Z_2(s)< -t\};t\ge 0\right).
		\end{equation*} Using the above convergence, the right-hand side of \eqref{eqn:xiiInteracting}, and Lemma \ref{lem:composition}(3) we see
		\begin{align*}
			\left( \overline{X}_n^{1,2}\circ \overline{U}_n^2 \circ \overline{X}_n^{2,1}(t);t\ge 0\right) \overset{M_1}{\longrightarrow} \left(\lambda_{12} \inf\{s: Z_2(s)< -\lambda_{12} t\};t\ge 0\right).
		\end{align*}
		The result now follows from \eqref{eqn:tildexij}, \eqref{eqn:xiiInteracting} and Lemma \ref{lem:composition}(1).
	\end{proof}
	
	\begin{lemma}
		Suppose $(\beta,\btheta,\lambda)\in \cI\times \R$ and let $\bzeta^{\beta,\btheta,\lambda}$ be as in \eqref{eqn:MCextremedef}. Then
		\begin{equation*}
			\inf\{u: W^{\beta,\btheta,\lambda}(u)< -t\} \overset{d}{=} J^{\bzeta^{\beta,\btheta,\lambda}}.
		\end{equation*} In particular, $Z^\ireg$ in \eqref{eqn:ZiIregDef} under Regime \ref{regime:interacting}
		\begin{equation*}
			\left(Z^{\ireg}(t);t\ge 0\right) = \left(W^{\beta_1,\btheta^\ireg, \lambda_{11}}(t);t\ge0\right)
		\end{equation*}
		where $\btheta^\ireg = \btheta_1\bowtie \lambda_{12}\bzeta^{\beta_2,\btheta_2,\lambda_{22}}.$
	\end{lemma}
	
	\begin{proof}
		For the first claim it follows from \cite[Corollary 16]{Aldous.97} (in the case $\btheta=\bzer$) and from Claim 6.4 in \cite{Martin:2017}. It can also be derived by the above analysis of $\overline{U}^2_n$. The second claim follows from Lemma \ref{lem:jumplevypart} and Lemma \ref{lem:summingtwothinned}.
	\end{proof}
	
	The next corollary follows from the previous lemma and Lemma \ref{lem:Zgood}.
	\begin{corollary}\label{cor:zireggoodpath}
		Under Regime \ref{regime:interacting}, the paths of $Z^\ireg$ in \eqref{eqn:ZiIregDef} are good with probability 1. In particular,
		\begin{equation*}
			t\mapsto     \inf\{u: Z_2(u)<-\lambda_{12}t\}
		\end{equation*} is continuous at each $l,r$ such that $(l,r)\in \EE_\infty(Z^\ireg).$
	\end{corollary}

	Just as in Lemma \ref{lem:vecforClassic} above, we can use the above corollary to establish the $\ell^2_\downarrow$ convergence of $\bcM_{p,1}$ in Regime \ref{regime:interacting}. In this case the matrix $R_n^{-1}\to I_{2\times 2}$ and we apply Corollary \ref{cor:techcor2} instead of Corollary \ref{cor:techcor}. Modulo those changes, the proof of Lemma \ref{lem:vecforClassic} remains largely unchanged. We leave these changes to the reader and conclude the following lemma.
	\begin{lemma}\label{lem:thmInteract}
		Theorem \ref{thm:interact} holds.
	\end{lemma}

	\subsection{Convergence of $V_n^\bip$ in regime \ref{regime:bipartite}}

	\begin{proposition}\label{prop:vconvbip}
		
		Suppose $\bW$ and $Q$ are as in Regime \ref{regime:bipartite}. Then jointly in the $M_1$ topology
		\begin{align*}
			&\left(c_nV_n^\bip(t);t\ge 0\right) \weakarrow (Z^\bip(t);t\ge 0)\\
			&\left( \widetilde{U}_n^2\circ \widetilde{X}^{2,1}_n(t);t\ge 0\right)\weakarrow (t;t\ge 0)
		\end{align*}
		where $Z^\bip$ is defined in \eqref{eqn:zbipdef}.
	\end{proposition}
	
	\begin{proof}
		Note that $q_{12} = \frac{1}{\sqrt{\sigma_2(\bw^1)\sigma(\bw^2)}}+\lambda_{12}+o(1)$ and so
		\begin{equation*}
			q_{12} = \sqrt{\frac{\sigma_2(\bw^1)}{\sigma_2(\bw^2)}} \frac{1}{\sigma_2(\bw^1)} + \lambda_{12}+o(1) = \frac{1}{\sigma_2(\bw^1)}+\lambda_{12}+\alpha
		\end{equation*}
		and similarly 
		\begin{equation*}
			q_{12} = \frac{1}{\sigma_2(\bw^2)} +\lambda_{12}-\alpha.
		\end{equation*}
		Therefore, by Corollary \ref{cor:Limic}, in the $M_1$ topology
		\begin{align*}
			&\left(c_n(\widetilde{X}^{2,1}_n(t) - t)\right) \weakarrow (W_1^{\beta_1,\btheta_1,\lambda_{12}+\alpha}(t);t\ge 0)\\
			&
			\left(c_n(\widetilde{X}^{1,2}_n(t) - t)\right) \weakarrow (W_2^{\beta_2,\btheta_2,\lambda_{12}-\alpha}(t);t\ge 0).
		\end{align*} Using Skorohod's representation theorem, we assume without loss of generality  that the convergence above holds a.s. instead of in distribution. By Lemma \ref{lem:weaklawFromCLT} the above convergence implies
		\begin{equation}\label{eqn:xjiconvbpfluct}
			(\widetilde{X}^{j,i}(t);t\ge 0) \overset{M_1}{\longrightarrow} (t;t\ge 0).
		\end{equation}
		Note that $\widetilde{X}^{i,i}(t) = -t + \frac{\lambda_{ii}+o(1)}{c_n} \widetilde{X}^{j,i}(t)$ and so 
		\begin{equation}\label{eqn:vnbip1}
			\left(c_n(\widetilde{X}^{i,i}(t)+t);t\ge 0\right) = \left((\lambda_{ii}+o(1))\widetilde{X}^{j,i}(t);t\ge 0\right) \overset{M_1}\longrightarrow (\lambda_{ii} ;t\ge 0).
		\end{equation} Let $\widetilde{L}_n^2(t) = -\inf_{s\le t} \widetilde{X}_n^{2,2}(s) = \sup_{s\le t}(-\widetilde{X}_n^{2,2}(s)).$ It follows from Lemma \ref{lem:sup} that
		\begin{align*}
			&(c_n(\widetilde{L}_n^2(t) - t);t\ge 0) \overset{M_1}{\longrightarrow} (\lambda_{22}t;t\ge 0)
		\end{align*}
		and so an application of Lemma \ref{lem:passageFluctuations} to $U_n^2$, the first passage time of $L_n^2$, gives
		\begin{align}\label{eqn:fluctForUBP}
			&(c_n(\widetilde{U}_n^2(t) - t);t\ge 0) \overset{M_1}{\longrightarrow} (\lambda_{22}t;t\ge 0).
		\end{align}

		Applying Corollary \ref{cor:threefunc}, we get
		\begin{align}
			\label{eqn:vnbip2}&\left( c_n(\widetilde{X}^{1,2}_n\circ \widetilde{U}_n^2 \circ \widetilde{X}^{2,1}_n(t)-t;t\ge 0\right) \\
			\nonumber&\qquad \overset{M_1}{\longrightarrow} \left(W_1^{\beta_1,\btheta_1,\lambda_{12}+\alpha}(t)+ \lambda_{22}t+W_2^{\beta_2,\btheta_2,\lambda_{12}-\alpha}(t);t\ge 0\right).
		\end{align}
		The convergence of $c_nV_n^\bip$ now follows from summing equations \eqref{eqn:vnbip1} and \eqref{eqn:vnbip2} and using the continuity of addition found in Lemma \ref{lem:composition}(1). 
		
		To get the second weak convergence result, we us Lemma \ref{lem:weaklawFromCLT} to see that \eqref{eqn:fluctForUBP} implies $\widetilde{U}_n^2 \to \operatorname{Id}$ locally uniformly and \eqref{eqn:xjiconvbpfluct} implies $\widetilde{X}^{2,1}\to \operatorname{Id}$ locally uniformly.
	\end{proof}
	
	By the same reasoning leading to Lemma \ref{lem:vecforClassic}, we can establish the following lemma. The details are omitted.
	\begin{lemma}\label{lem:vecforBP}
		Suppose $\bW =\bW^{(n)}$ and $Q = Q^{(n)}$ are as in Regime \ref{regime:bipartite}. Let $(\bcM_l^{(n)};l\ge 1)$ be the vector-valued component masses of $\G(\bW,Q)$ ordered so that $\bcM_{l,1}^{(n)}\ge \bcM_{l,2}^{(n)}\ge\dotsm$. 
		Then 
		\begin{equation*}
			\left(\bcM_{p,j}^{(n)};l\ge 1\right) \weakarrow \left(\zeta_p^\bip;p\ge 1\right)\qquad\begin{array}{l}
				\textup{ in }\ell^2_\downarrow\textup{ for }j =1\textup{ and }\\
				\textup{pointwise }\textup{for }j =2\\
			\end{array}
		\end{equation*}
		where $\bzeta^\bip$ is as in Theorem \ref{thm:classic}.
	\end{lemma}
	
	\section{Proof of Main Results}\label{sec:proofsofMain}
	
	The general approach to proving scaling limits in $\ell^2_\downarrow$ is to use the general size-biased point process approach of Aldous \cite{Aldous.97}. As we have seen in Lemmas \ref{lem:vecforClassic} and \ref{lem:vecforBP}, this works quite well for the marginal convergence of the mass of type 1 vertices in each connected component, but this does not hold jointly in the $\ell^2$ topology. 
	
	\subsection{Preliminary Lemmas}
	In order to extend the aforementioned lemmas, we establish with some elementary yet general results about convergence of sequences of vectors in $\ell^2(\N\to\R^2_+)$. We will denote generic elements of $\ell^2(\N\to \R^2_+)$ by $\bz = (\vec{z}_l;l\ge 1)$ where $\vec{z}_l = (x_l,y_l)$. Denote by $\ORD_1(\bz) = (\vec{z}_{\tau_l};l\ge 1)$ as the rearrangement of $\{\{\vec{z}_l;l\ge 1\}\}$ in the unique way that $l\mapsto x_{\tau_l}$ is non-decreasing and if $x_{\tau_l} = x_{\tau_r}$ for $l<r$ then $\tau_l<\tau_r$. Similarly denote $\ORD_2(\bz)$ as the rearrangement so that the ``$y$'' coordinate is decreasing. Let $\pi_1(\bz) = (x_l;l\ge 1)$ and $\pi_2 (\bz) = (y_l;l\ge 1)$.
	
	We begin with a simple deterministic lemma. 
	\begin{lemma}
		Suppose that $\bz^{(n)}$ is a sequence in $\ell^2(\N\to\R^2_+)$. Let $a_n\to 1$ and $b_n\to 0$ are sequences of real numbers and that $\bw = (w_l;l\ge 1)\in\ell^2_\downarrow.$ If
		\begin{enumerate}
			\item $\{\pi_2(\bz^{(n)});n\ge 1\}$ is pre-compact in $\ell^2$;
			\item As $n\to\infty$ 
			\begin{equation*}
				\ORD(\pi_1(\bz)) \longrightarrow \bw\textup{ in }\ell^2_\downarrow.
			\end{equation*}
		\end{enumerate}
		Then $\ORD(a_n x_l^{(n)}+b_n y_l^{(n)};l\ge 1)\to \bw$ in $\ell^2$.
	\end{lemma}
	\begin{proof}
		Without loss of generality we assume that $\vec{z}_l^{(n)}$ is re-ordered so that $\pi_1(\bz^{(n)}) = \ORD(\pi_1(\bz^{(n)}))$. It is easy to see that $\{(a_nx_l^{(n)}+b_n y_l^{(n)};l\ge 1): n\ge 1\}$ is pre-compact in $\ell^2$ and also that \begin{equation*}
			\sup_{n,l} y_l^{(n)}<\infty.
		\end{equation*} It follows that for each fixed $l\ge 1$ that $a_nx_l^{(n)} + b_n y_l^{(n)} \to w_l$, where we recall that we have reorded $\bz^{(n)}$ so that $x_l^{(n)}\to w_l$. 
		
		As point-wise convergence and relative compactness implies $\ell^2$-convergence, we obtain 
		\begin{equation*}
			(a_nx_l^{(n)}+b_n y^{(n)}_l;l\ge 1)\longrightarrow \bw\qquad \textup{ in }\ell^2.
		\end{equation*}
		The result now follows from recognizing $\ORD$ is a contraction on $\ell^2$. That is for any two vectors $\bw$ and $\bw'$ in $\ell^2$:
		\begin{equation*}
			\|\ORD(\bw)-\ORD(\bw')\|^2_{\ell^2} \le \|\bw-\bw'\|^2_{\ell^2.}
		\end{equation*}
	\end{proof}
	
	The next lemma is a trivial and its proof is omitted.
	\begin{lemma}
		Suppose that $\bz^{(n)}, n\ge 1$ and $\bz$ are elements $\ell^2(\N\to\R^2_+)$. Let $A_n\to A\in\R^{2\times 2}$. If $\bz^{(n)}\to \bz$ in $\ell^2(\N\to\R_+^2)$ then 
		\begin{equation*}
			(A_n \vec{z}_l^{(n)};l\ge 1) \longrightarrow \left(A \vec{z}_l;l\ge 1\right)\qquad\textup{ in }\ell^2(\N\to\R^2_+).
		\end{equation*} 
	\end{lemma}
	
	The next lemma is key to proving Theorems \ref{thm:classic} and \ref{thm:bipartite}. 
	\begin{lemma}\label{lem:margtowhole}
		Suppose that $\bz^{(n)}$ is a sequence in $\ell^2(\N\to\R^2_+)$, $\alpha,\beta>0$, and $\bx, \by\in \ell^2_\downarrow$. If
		\begin{enumerate}
			\item $\ORD(\pi_1(\bz^{(n)})) \to \bx$ in $\ell^2$;
			\item $\ORD(\pi_2(\bz^{(n)})) \to \by$ in $\ell^2$;
			\item $\ORD_1(\bz^{(n)}) \to \left(\begin{bmatrix}
				1\\\alpha
			\end{bmatrix}x_l;l\ge 1\right)$ in the product topology;
			\item $\ORD_2(\bz^{(n)}) \to \left(\begin{bmatrix}
				\beta\\1
			\end{bmatrix}y_l;l\ge 1\right)$ in the product topology; and
			\item $y_l>0$ for all $l\ge 1.$
		\end{enumerate}
		Then $\ORD_1(\bz^{(n)}) \to \left(\begin{bmatrix}
			1\\\alpha
		\end{bmatrix}x_l;l\ge 1\right)$ in $\ell^2(\N\to \R^2_+).$
	\end{lemma}
	
	\begin{proof}
		
		Without loss of generality, we can suppose that $\bz^{(n)} = \ORD_1(\bz^{(n)})$. It suffices to show
		\begin{equation}\label{eqn:veczconv}
			\lim_{k\to\infty} \limsup_{n\to\infty} \sum_{l>k} \left\|\vec{z}_l^{(n)} - x_l \begin{bmatrix}
				1\\\alpha
			\end{bmatrix}\right\|_{1}^2  = 0.
		\end{equation}
		To do this first note that by the equivalence of norms on $\R^2$ that there is a constant $C\in(0,\infty)$ such that
		\begin{align*}
			\sum_{l>k}& \left\| \vec{z}_l^{(n)} - x_l \begin{bmatrix}
				1\\\alpha
			\end{bmatrix}\right\|^2_1 \le C   \sum_{l>k} \left\| \vec{z}_l^{(n)} - x_l \begin{bmatrix}
				1\\\alpha
			\end{bmatrix}\right\|^2_2 \\
			&\le C \sum_{l>k } \left((x_l^{(n)} - x_l)^2+ (y_l^{(n)})^2 +\alpha x_l^2 \right).
		\end{align*}
		In the last inequality we used $(a-b)^2\le a^2+b^2$ for $a,b\ge 0$.
		As $\pi_1(\bz^{(n)}) \to \bx$ in $\ell^2_\downarrow$ and $\bx\in\ell^2$, equation \eqref{eqn:veczconv} follows from showing
		\begin{equation*}
			\lim_{k\to\infty} \limsup_{n\to\infty} \sum_{l>k} (y_l^{(n)})^2 = 0.
		\end{equation*}
		
		We will now show that for all $\eps>0$, we can find a $k = k(\eps)$ such that 
		\begin{equation}\label{eqn:taily3}
			\limsup_{n\to\infty}  \sum_{l>k} (y_l^{(n)})^2\le \eps.
		\end{equation}Let us write $\bx^{(n)} = \pi_1(\bz^{(n)}), \by^{(n)} = \pi_2(\bz^{(n)})$, $\widetilde{\by}^{(n)} = \ORD(\by^{(n)})$ and $\widetilde{\bx}^{(n)} = \pi_1(\ORD_2(\bz^{(n)}))$. Note that since $\widetilde{\by}^{(n)}\longrightarrow \by$ in $\ell^2$ it follows that $\{\widetilde{\by}^{(n)};n\ge 1\}$ is pre-compact and hence 
		\begin{equation*}
			\lim_{\delta\downarrow 0} \sup_{n\ge 1} \sum_{l=1}^\infty ({y}^{(n)}_l1_{[{{y}}_l^{(n)}\le \delta]})^2 =    \lim_{\delta\downarrow 0} \sup_{n\ge 1} \sum_{l=1}^\infty (\widetilde{y}^{(n)}_l1_{[{\widetilde{y}}_l^{(n)}\le \delta]})^2 = 0.
		\end{equation*} This follows from Dini's theorem. Hence, there is a $\delta = \delta_\eps>0$ such that 
		\begin{equation}\label{eqn:taily2}
			\sup_{n\ge 1} \sum_{l=1}^\infty\left( y^{(n)}_l\right)^2 1_{[y_l^{(n)}\le\delta]}\le \eps.
		\end{equation}
		Note that as $\by\in\ell^2_\downarrow$, we can find an $m = m(\eps)$ such that 
		\begin{equation*}
			y_l \le \delta/4\qquad\forall l> m.
		\end{equation*}
		Similarly, we can find a $k = k(\eps) \ge m$ such that 
		\begin{equation*}
			x_l \le \frac{\beta y_m}{4}\qquad \forall l > k.
		\end{equation*}
		As $\bx^{(n)}\to \bx$ and $\widetilde{\by}^{(n)}\to \by$ in $\ell^2_\downarrow$, it follows that $x_p^{(n)}\to x_p$ and $\widetilde{y}_p^{(n)}\to y_p$ for all $p\le k$. Recall, both $\bx^{(n)}$ and  $\widetilde{\by}^{(n)}$ are ordered. Therefore, we can find an $n_1$ sufficiently large such that for all $n\ge n_1$
		\begin{align}\label{eqn:tildeyl}
			&\textup{ if}\,\,\, x^{(n)}_l \ge \frac{\beta y_m}{2}\,\, \textup{   then  
			} \,\,l\le k &&\textup{and}&&\textup{ if}\,\,\widetilde{y}^{(n)}_p \ge \delta/2\,\, \textup{   then  
			} \,\,p\le m.
		\end{align}
		
		We claim that there exists an $n_2\ge n_1$ sufficiently large such that
		\begin{equation}\label{eqn:taily}
			\sum_{l>k} (y_l^{(n)})^2 \le \sum_{l=1}^\infty \left(y_l^{(n)}\right)^2 1_{[y_l^{(n)}\le \delta]}  \qquad\forall n\ge n_2.
		\end{equation}
		To show this, suppose that it does not hold. That means there exists a sequence $N_j\to \infty$ with $N_1\ge n_1$ such that
		\begin{equation*} 
			\sum_{l>k} \left(y_l^{(N_j)}\right)^2 > \sum_{l=1}^\infty \left(y_l^{(N_j)}\right)^2 1_{[y_l^{(n)}\le \delta]}\qquad \forall j\ge 1.
		\end{equation*}
		But this implies that for each $j\ge 1$ there exists a $p_j>k$ such that $y_{p_j}^{(N_j)} > \delta$. As $\widetilde{\by}^{(N_j)}= \ORD(\by^{(N_j)})$, we conclude from \eqref{eqn:tildeyl} that 
		\begin{equation*}
			y_{p_j}^{(N_j)} \in \{ \widetilde{y}^{(N_j)}_l : l\le m\}.
		\end{equation*} Therefore,
		\begin{equation*}
			\begin{bmatrix}
				x_{p_j}^{(N_j)}\\
				y_{p_j}^{(N_j)}
			\end{bmatrix} \in \left\{\begin{bmatrix}
				\widetilde{x}^{(N_j)}_l \\
				\widetilde{y}^{(N_j)}_l 
			\end{bmatrix}; l\le m\right\}
		\end{equation*}
		This implies that
		\begin{equation*}
			\liminf_{j\to\infty} x^{(N_j)}_{p_j} \ge \liminf_{j\to\infty} \min_{l\le m}\widetilde{x}_l^{(N_j)} = \min_{l\le m} \lim_{j\to\infty} \widetilde{x}_l^{(N_j)} = \beta y_m> \frac{\beta y_m}{2}
		\end{equation*}
		where the strict inequality follows from $y_l>0$ for all $l$. But also, as $p_j > k$ we know from \eqref{eqn:tildeyl} that $x_{p_j}^{(N_j)}< \beta y_m/2$. This gives the desired contradiction and \eqref{eqn:taily} holds. By \eqref{eqn:taily2} and \eqref{eqn:taily}, \eqref{eqn:taily3} holds and the result follows.
	\end{proof}

	\subsection{Proofs of Theorems \ref{thm:classic} and \ref{thm:bipartite}}
	
	The proofs of Theorems \ref{thm:classic} and \ref{thm:bipartite} are nearly identical from this point onwards. We restrict our attention to Theorem \ref{thm:classic}.
	
	To begin, note that by Lemma \ref{lem:vecforClassic} that jointly in $j\in[2]$
	\begin{equation*}
		\pi_j\ORD_1 (\bcM_l^{(n)};l\ge 1) \weakarrow (\bzeta_p^\creg u_j^\creg)
	\end{equation*} where the topology is point-wise for $j =2$ and in $\ell^2_\downarrow$ for $j=1$. By reversing the roles of the type 1 and type 2 vertices, we see that for some element $\bzeta' = (\zeta_p';p\ge 1)\in \ell^2_\downarrow$ and some $\vec{v} = (v_1,v_2)\in (0,\infty)^2$ it holds that
	\begin{equation*}
		\pi_j\ORD_2 (\bcM^{(n)}_l;l\ge 1) \weakarrow (\zeta'_pv_j)
	\end{equation*} 
	where the convergence is pointwise for $j = 1$ and in $\ell^2_\downarrow$ for $j = 2$.
	
	Note that $\ell^2_\downarrow$ and $\R_+^\infty$ equipped with the topology of $\ell^2$ convergence and point-wise convergence (respectively) are Polish. Hence, a standard tightness argument implies
	\begin{align*}
		\bigg\{\Big(&\pi_1\big(\ORD_1 (\bcM_l^{(n)};l\ge 1)\big), \pi_2\big(\ORD_1 (\bcM_l^{(n)};l\ge 1)\big), \\
		& \qquad \pi_1\big(\ORD_2 (\bcM_l^{(n)};l\ge 1)\big),\pi_2\big(\ORD_2 (\bcM_l^{(n)};l\ge 1)\big)\Big); n\ge 1\bigg\} 
	\end{align*} 
	is tight in $\ell^2_\downarrow\times \R^\infty_+\times \R_+^\infty \times \ell^2_\downarrow.$ As this space is Polish, each subsequence of $n$ (say $n_r$) has a further subsequential weak limit (along say $n'_r$). By Skorohod's representation theorem, we can suppose that along this subsequence a.s.
	\begin{align*}
		\Big(&\pi_1\big(\ORD_1 (\bcM_l^{(n'_r)};l\ge 1)\big), \pi_2\big(\ORD_1 (\bcM_l^{(n'_r)};l\ge 1)\big), \\
		& \qquad \pi_1\big(\ORD_2 (\bcM_l^{(n'_r)};l\ge 1)\big),\pi_2\big(\ORD_2 (\bcM_l^{(n'_r)};l\ge 1)\big)\Big)\\
		&\longrightarrow\left( (\zeta_p^\creg u_1)_{p\ge 1}, (\zeta_p^\creg u_2)_{p\ge 1}, (\zeta'_pv_1)_{p\ge 1}, (\zeta'_pv_2)_{p\ge 1}\right)
	\end{align*} in $\ell^2_\downarrow\times\R_+^\infty\times\R_+^\infty\times \ell_\downarrow^2.$ Now just apply Lemma \ref{lem:margtowhole} and the observation that all subsequential limits are the same.

	\bibliographystyle{abbrv}

\end{document}